\documentclass[11pt, a4paper]{amsart}
\usepackage{amssymb}
\usepackage{amsmath}
\usepackage{amssymb}
\usepackage{amsthm}
\usepackage{bbm}
\usepackage{bm}
\usepackage{mathtools}
\usepackage{mathrsfs}
\usepackage[T1]{fontenc}
\usepackage[usenames,dvipsnames,svgnames,table]{xcolor}
\usepackage{esint}
\allowdisplaybreaks

\usepackage{enumitem}
\setenumerate{label={\rm (\alph{*})}}
\usepackage[colorlinks=true,citecolor=black,urlcolor=black,
linkcolor=black]{hyperref}

\usepackage{graphicx}
\usepackage{tikz}

\usepackage{geometry}
\makeatletter
\newcommand*\bigcdot{\mathpalette\bigcdot@{.5}}
\newcommand*\bigcdot@[2]{\mathbin{\vcenter{\hbox{\scalebox{#2}{$\m@th#1\bullet$}}}}}
\makeatother

\numberwithin{equation}{section}

\newtheorem{theorem}{Theorem}[section]
\newtheorem{lemma}[theorem]{Lemma}
\newtheorem{proposition}[theorem]{Proposition}

\newtheorem{definition}[theorem]{Definition}

\newtheorem{remark}[theorem]{Remark}
\newtheorem{example}[theorem]{Example}

\numberwithin{equation}{section}

\usepackage{booktabs}

\normalsize

\def\XXint#1#2#3{{\setbox0=\hbox{$#1{#2#3}{\int}$ }
\vcenter{\hbox{$#2#3$ }}\kern-.6\wd0}}

\newcommand{\bd}{\operatorname{BD}}
\newcommand{\bv}{\operatorname{BV}}

\newcommand{\di}{\operatorname{div}}

\newcommand{\dif}{\operatorname{d}\!}

\newcommand{\singspt}{\mathrm{sing\,spt\,}}

\newcommand{\spt}{\operatorname{spt}}

\newcommand{\R}{\mathbb{R}}
\newcommand{\A}{\mathbb{A}}
\newcommand{\B}{\mathbb{B}}

\newcommand{\locc}{\operatorname{loc}}

\newcommand{\ball}{B}

\newcommand{\sobo}{\operatorname{W}}

\newcommand{\lebe}{\operatorname{L}}

\newcommand{\hold}{\operatorname{C}}

\newcommand{\D}{D}
\newcommand{\curl}{\operatorname{curl}}
\renewcommand{\leq}{\leqslant}

\newcommand{\imag}{\operatorname{i}}
\newcommand{\id}{\operatorname{Id}}

\newcommand{\e}{\operatorname{e}}

\renewcommand{\di}{\operatorname{div}}
\newcommand{\bmo}{\operatorname{BMO}}

\newcommand{\asym}{\operatorname{asym}}

\newcommand{\lin}{\operatorname{Lin}}

\renewcommand{\e}{\operatorname{e}}

\newcommand{\taylor}{\operatorname{t}}
\newcommand{\Taylor}{\operatorname{T}}
\usepackage{tikz-cd}
\AtEndDocument{\bigskip{\footnotesize%
  \textsc{Andrew Wiles Building, University of Oxford, Woodstock Rd, Oxford OX2 6GG.} \par  
  \noindent\textit{E-mail address}, B.~Rai\c{t}\u{a}: \texttt{raita@maths.ox.ac.uk}
}}

\begin{document}
\title[Critical $\lebe^p$--differentiability of $\bv^\A$--maps]{Critical $\lebe^p$--differentiability of $\bv^\A$--maps\\ and canceling operators}
\author[B. Rai\c{t}\u{a}]{Bogdan Rai\c{t}\u{a}}
\thanks{This work was supported by Engineering and Physical Sciences Research Council Award EP/L015811/1. This project has received funding from the European Research Council (ERC) under the European Union's Horizon 2020 research and innovation programme under grant agreement No 757254 (SINGULARITY)}
\begin{abstract}
We give a generalization of \textsc{Dorronsoro}'s Theorem \cite{Dor} on critical $\lebe^p$--Taylor expansions for $\bv^k$--maps on $\R^n$, i.e., we characterize homogeneous linear differential operators $\A$ of $k$--th order such that $D^{k-j}u$ has $j$--th order $\lebe^{n/(n-j)}$--Taylor expansion a.e. for all $u\in\bv^\A_{\locc}$ (here $j=1,\ldots, k$, with an appropriate convention if $j\geq n$). The space $\bv^\A_{\locc}$, a single framework covering $\bv$, $\bd$, and $\bv^k$, consists of those locally integrable maps $u$ such that $\A u$ is a Radon measure on $\R^n$.

For $j=1,\ldots,\min\{k, n-1\}$, we show that the $\lebe^p$--differentiability property above is equivalent with \textsc{Van Schaftingen}'s \emph{elliptic and canceling} condition for $\A$ \cite{VS}. For $j=n,\ldots, k$, ellipticity is necessary, but cancellation is not. To complete the characterization, we determine the class of elliptic operators $\A$ such that the estimate
\begin{align}\tag{1}\label{eq:abs}
\|D^{k-n}u\|_{\lebe^\infty}\leq C\|\A u\|_{\lebe^1}
\end{align}
holds for all vector fields $u\in\hold^\infty_c$. Surprisingly, the (computable) condition on $\A$ such that \eqref{eq:abs} holds is \emph{strictly} weaker than cancellation. 

The results on $\lebe^p$--differentiability can be formulated as sharp pointwise regularity results for overdetermined elliptic systems
\begin{align*}
\A u=\mu,
\end{align*}
where $\mu$ is a Radon measure, thereby giving a variant for the limit case $p=1$ of a Theorem of \textsc{Calder\'on} and \textsc{Zygmund} in \cite{CZ1} which was not covered before.
\end{abstract}
\maketitle
\section{Introduction}
\subsection{First order operators}
Approximate differentiability properties of weakly differentiable functions are well understood \cite[Ch.~6]{EG}. In particular, \textsc{Calder\'on} and \textsc{Zygmund} proved in \cite{CZ1,CZ2} that a map $u\in\bv_{\locc}(\R^n,\R^m)$ is $\lebe^{n/(n-1)}$--differentiable at $\mathscr{L}^n$--a.e. $x\in\R^n$, i.e., there exists a matrix $\nabla u(x)$ (called \emph{approximate gradient}) such that 
\begin{align*}
\|R^1_x u\|_{\lebe^{n/(n-1)}(\ball_r(x))}=o(r^n)\quad\text{ as }r\downarrow0.
\end{align*}
Here $R^1_x u$ denotes the first order Taylor remainder
\begin{align*}
(R^1_x u)(y)\coloneqq u(y)-u(x)-\nabla u(x)(y-x)
\end{align*}
for all $y\in\R^n$. It was recently proved in \cite[Cor.~2.6]{JS} and independently in \cite{GR2} that the same property holds true of the space $\bd$ of maps of bounded deformation, building on the results in \cite{ABC,ACDM,Haj}. The space $\bd$, arising naturally in plasticity problems \cite{AG,FS,ST}, is defined as the space of integrable maps $u\colon\R^n\rightarrow\R^n$ such that the distributional symmetrized gradient $\mathcal{E}u\coloneqq (\D u+\D u^\mathrm{t})/2$ is a bounded measure. The $\lebe^{n/(n-1)}$--differentiability of $\bd_{\locc}$--maps cannot be extrapolated from the $\bv$ result since by Ornstein's Non--inequality \cite{CFM,KirKri,Ornstein}, there are $\bd$--maps for which the full distributional gradient is not a bounded measure. Instead, the view point taken in \cite{ABC,Haj} is that test functions can be retrieved from their symmetrized gradient via convolution. In analogy with the fact that the Riesz potential $I_1$ is bounded from $\lebe^1$ into $\lebe^{n/(n-1)}_{\text{weak}}$ (and \emph{not} into $\lebe^{n/(n-1)}$), the harmonic analysis techniques used in \cite{ABC,Haj} yield $\lebe^p$--differentiability of $\bd_{\locc}$--maps for all $1\leq p<n/(n-1)$, but cannot easily cover the critical case. The idea used in \cite{GR2} was to use the embedding $\bd\hookrightarrow\lebe^{n/(n-1)}$ \cite[Prop.~1.2]{ST} in the form of a Poincar\'e--Sobolev--type inequality. In fact, the arguments in \cite{GR2} cover a class of differential operators that contains $\mathcal{E}$. To be precise, we define the space $\bv^\A(\Omega)$ as the space of $u\in\lebe^1(\Omega,V)$ such that $\A u\in\mathcal{M}(\Omega,W)$ is a bounded measure, for open sets $\Omega\subset\R^n$, where the differential operator $\A$ is defined, independently of coordinate choice, by
\begin{align}\label{eq:A}
\A u\coloneqq A( D u) 
\end{align}
for smooth maps $u\colon\R^n\rightarrow V$, a linear map $A\in\lin(V\otimes\R^n,W)$, and finite dimensional normed vector spaces $V,W$. For simplicity of exposition, we assume in addition that $\R^n$, $V$, $W$ are all equipped with fixed inner products. This is no restriction, as we will explain later. Although in this section we assume that $n>1$, the embedding presented in Theorem~\ref{thm:main_infty} will enable us to also cover the case $n=1$, which is ruled out for similar results, e.g., \cite[Thm.~3.4]{ABC} or \cite[Thm.~1.1]{GR2}.

The main result in \cite{GR2} states that if the solution space of $\A u=0$ in $\mathscr{D}^\prime(\R^n,V)$ is finite dimensional (when we say that $\A$ \emph{has FDN}), then any map in $\bv^\A_{\locc}$ is $\lebe^{n/(n-1)}$--differentiable $\mathscr{L}^n$--a.e.. Of course, this covers the physically relevant case $\bv^\A=\bd$, since $\mathcal{E}u=0$ is satisfied only by rigid deformations \cite[Thm.~3.2]{HN}. It is however natural to ask the mathematically relevant question to determine the class of operators $\A$ such that all maps in $\bv^\A_{\locc}$ are $\lebe^{n/(n-1)}$--differentiable $\mathscr{L}^n$--a.e.. What we know is that such a class contains the FDN operators and it is not difficult to see that it is contained in the class of elliptic operators. Indeed, we will establish this in Lemma~\ref{lem:nec_ell_appdiff} (recall that an operator $\A$ as in \eqref{eq:A} is (overdetermined) \emph{elliptic} if the symbol map $V\ni v\mapsto A(v\otimes\xi)\eqqcolon \A[\xi]v$ is injective). It was already shown in \cite[Rk.~3.2]{GR2} that ellipticity is not sufficient for differentiability. On the other hand, as explained in \cite{BDG,GR}, the FDN condition is equivalent with boundary regularity, respectively, of traces of maps in $\bv^\A(\Omega)$ and of solutions of $\A u=0$ in $\mathscr{D}^\prime(\Omega,V)$. It does not a priori seem likely that FDN is equivalent to a pointwise regularity property, as is $\lebe^{n/(n-1)}$--differentiability. In light of the proof of \cite[Thm.~1.1]{GR2}, it is however feasible that $\lebe^{n/(n-1)}$--differentiability of $\bv^\A_{\locc}$--maps is equivalent with a homogeneous Sobolev--type embedding. Indeed, we will prove in Lemma~\ref{lem:suff_EC} that the embedding
\begin{align}\label{eq:VS_appdiff}
\|u\|_{\lebe^{n/(n-1)}}\lesssim\|\A u\|_{\lebe^1}
\end{align}
for all $u\in\hold^\infty_c(\R^n,V)$ is sufficient to prove the required $\lebe^{n/(n-1)}$--differentiability. The precise conditions on $\A$ for \eqref{eq:VS_appdiff} to hold, namely elliptic and \emph{canceling} (EC), were established in \cite{BB07,VS}. The canceling condition, i.e.,
\begin{align*}
\bigcap_{\xi\in \mathbb{S}^{n-1}}\mathrm{im\,}\A[\xi]=\{0\},
\end{align*}
was introduced by \textsc{Van Schaftingen} in \cite{VS} and can indeed be interpreted as an interior regularity (integrability) condition, as we will discuss in Lemma~\ref{lem:canc_char}. In fact, in view of the Lorentz space embedding \cite[Thm.~8.5]{VS} (see also \cite{Peetre_lore,Alvino,Hunt,Tartar_lorentz,Tartar} for previous results on Lorentz--Sobolev embeddings), we say that a map $u$ is $\lebe^{n/(n-1),q}$--differentiable at $x$ if
\begin{align*}
\|R_x u\|_{\lebe^{n/(n-1),q}(\ball_r(x))}=o(r^n)\quad\text{ as }r\downarrow0, 
\end{align*}
where $1\leq q\leq \infty$, then we can obtain the following refined statement:
\begin{theorem}\label{thm:main_k=1}
Let $\A$ be as in \eqref{eq:A}, $1<q<\infty$, $n>1$. Then $\A$ is EC if and only if all maps in $\bv^\A_{\locc}$ are $\lebe^{n/(n-1),q}$--differentiable $\mathscr{L}^n$--a.e.
\end{theorem}
Our proof of sufficiency of EC is elementary and somewhat novel, in the sense that all other proofs of $\lebe^p$--differentiability of $\bv$-- or $\bd$--maps using Poincar\'e(--type) inequalities or Sobolev(--type) embeddings that we traced in the literature cannot easily be adjusted to rely on \eqref{eq:VS_appdiff} only (cp. \cite{ACDM,AFP,EG,GR2}). We also see no alternative way to link the equivalent condition EC to the critical differentiability, other than by use of \eqref{eq:VS_appdiff}. It is important to mention that it was proved in \cite[Sec.~3]{GR} that EC is strictly weaker than FDN, so the sufficiency part of Theorem~\ref{thm:main_k=1} is not a vacuous extension of \cite[Thm.~1.1]{GR2}. Our proof of Theorem~\ref{thm:main_k=1} can be used to show that at the endpoint $q=\infty$ the claim is equivalent with ellipticity of $\A$ alone. In view of \cite[Open Prob.~8.3]{VS}, we do not know whether the critical case $q=1$ can be achieved, except in the case $\A =\D$. It is known that the embedding
\begin{align*}
\dot{\sobo}{^{1,1}}(\R^n,\R^N)\hookrightarrow\lebe^{n/(n-1),1}(\R^n,\R^N)
\end{align*}
holds \cite{Dor,Stoly,Tartar}, and our method implies the following fact which we could not trace in the literature, but is probably known to experts::
\begin{proposition}
Let $u\in\bv_{\locc}(\R^n,\R^N)$. Then $u$ is $\lebe^{n/(n-1),1}$--differentiable at $\mathscr{L}^n$--a.e. $x\in\R^n$.
\end{proposition}
\subsection{A new $\lebe^\infty$--embedding}
The result of Theorem~\ref{thm:main_k=1} extends to higher order homogeneous operators, by which we mean
\begin{align}\label{eq:Ak}
\A u\coloneqq A(D^k u)
\end{align}
for smooth maps $u\colon\R^n\rightarrow V$ and $A\in\lin(V\odot^{k}\R^n,W)$. Here $V\odot^k\R^n$ denotes the vector space of symmetric, $V$--valued, $k$--linear maps on $\R^n$. More precisely, in order to investigate the $\lebe^p$--differentiability properties of the derivatives of maps in $\bv^\A_{\locc}$, we will use the following inequality for EC operators
\begin{align}\label{eq:VS_j}
\|\D^{k-j}u\|_{\lebe^{n/(n-j)}}\lesssim\|\A u\|_{\lebe^1}
\end{align}
for $u\in\hold^\infty_c(\R^n,V)$. This follows by \cite[Thm.~1.3]{VS} and iterative application of the Gagliardo--Nirenberg--Sobolev Inequality \emph{only} for $j=1\ldots \min\{k,n-1\}$. If $k\geq n$, the boundedness of $D^{k-n}u$ cannot be inferred from \eqref{eq:VS_j} for $j=n-1$ since $\dot{\sobo}{^{1,n}}$ does not embed into $\lebe^\infty$, but $\bmo$. This phenomenon is already observed in the case $\A=D^n$, in which case one has to use a different method\footnote{Simply the Fundamental Theorem of Calculus in all coordinate directions, in this case.} to show that $\dot{\sobo}{^{n,1}}$ does indeed embed in $\lebe^\infty$ ($\hold_0$, even). In more generality, it was shown by \textsc{Bousquet} and \textsc{Van Schaftingen} in \cite[Thm.~1.3]{BVS} that if $k\geq n$ and $\A$ is EC, then
\begin{align}\label{eq:VS_k=n}
\|\D^{k-n}u\|_{\lebe^\infty}\lesssim\|\A u\|_{\lebe^1}
\end{align}
for all $u\in\hold^\infty_c(\R^n,V)$. The definition of EC is formally the same as in the first order case, with the obvious modification of the definition of the symbol map $\A[\xi]v\coloneqq A(v\otimes^k\xi)$. Here $v\otimes^k\xi\coloneqq v\otimes\xi\otimes\xi\otimes\ldots\otimes\xi$, where the exterior product is taken $k$ times.

So far, little is known about the necessity of the EC condition for \eqref{eq:VS_k=n}. Since for our $\lebe^p$--differentiability claims ellipticity is necessary (see Lemmas~\ref{lem:nec_ell_appdiff}, \ref{lem:nec_ell_k}), we will assume it. As for cancellation, it may not be necessary, as the simple example $\A=D$ if $n=1$ presented in \cite{BVS} suggests. One of the main results of this paper is to show that, somewhat surprisingly, \eqref{eq:VS_k=n} is equivalent to a substantially weaker new condition:
\begin{theorem}\label{thm:main_infty}
Suppose that $\A$ as in \eqref{eq:Ak} is elliptic, $k\geq n\geq1$. Then
\begin{align*}
\|D^{k-n}u\|_{\lebe^\infty(\R^n,V\odot^{k-n}\R^n)}\lesssim \|\A u\|_{\lebe^1(\R^n,V)}
\end{align*}
holds for all $u\in\hold^\infty_c(\R^n,V)$ if and only if 
\begin{align}\tag{WC}\label{eq:mistery_cond2}
\mathcal{L}w\coloneqq\int_{\mathbb{S}^{n-1}}\A^\dagger[\xi]w\otimes^{k-n}\xi\dif\mathscr{H}^{n-1}(\xi)=0~\text{ for all }~w\in\bigcap_{\xi\in\mathbb{S}^{n-1}}\mathrm{im\,}\A[\xi],
\end{align}
where $\mathcal{L}\in\lin(W,V\odot^{k-n}\R^n)$ is defined on the entire $W$ by the same formula.
\end{theorem}
We use the notation $\A^\dagger[\xi]\coloneqq(\A^*[\xi]\A[\xi])^{-1}\A^*[\xi]$ for $\xi\neq0$ and $\coloneqq0$ otherwise, where ``$\,^*\,$'' denotes the adjoint. $\A^\dagger[\cdot]$ is well--defined by the ellipticity assumption.

The importance of Theorem~\ref{thm:main_infty} for $\lebe^1$--estimates is that, under the ellipticity assumption, it settles the investigation of Sobolev estimates of the type \eqref{eq:VS_j} with the endpoint case $j=n$. In the absence of ellipticity, there is no systematic method to obtain even weak--type estimates, which should be a simple baseline if $j<n$; only few examples have been considered, see e.g., \cite[Prop.~5.4]{VS} or \cite{PVS}.

We next discuss condition~\eqref{eq:mistery_cond2}, which is obviously weaker than the canceling condition, thereby recovering the result in \cite[Thm.~1.3]{BVS}. By a simple homogeneity consideration, one immediately notes that condition~\eqref{eq:mistery_cond2} is automatically satisfied if $n$ is odd (in which case $\mathcal{L}\equiv0$). In Section~\ref{sec:examples_appdiff}, we will give examples to show that in both even and odd dimensions, condition~\eqref{eq:mistery_cond2} is strictly weaker than the canceling condition. In Proposition~\ref{prop:char_mist_cond}, we will give an analytic characterization which highlights why the discrepancy is possible. On the other hand, Proposition~\ref{prop:char_mist_cond} reveals a common theme shared by \eqref{eq:VS_j} and \eqref{eq:VS_k=n}, which is in sharp contrast to Ornstein's Non--inequality. We will expand on this point in Section~\ref{sec:char_mist_cond}.

We next compare Theorem~\ref{thm:main_infty} with other instances when quantitative restrictions on the data of a partial differential equation imply everywhere continuity of the solution. An observation in this direction was made by \textsc{Stein} in \cite{Stein_note}, where it is shown, as a recovery of the fact that $\dot{\sobo}{^{1,n}}\not\hookrightarrow \hold_0$, that
\begin{align*}
\D u\in\lebe^{n,1}(\R^n,\R^n)\implies u\in\hold_0(\R^n),
\end{align*}
up to the choice of a representative. In the same vein, one can show that for each $l=1,\ldots,n$, we have
\begin{align}\label{eq:pizda}
D^l u\in\lebe^{n/l,1}\implies u\in\hold_0(\R^n),
\end{align}
(see for instance \cite{KorKri}, where a detailed analysis of this and related phenomena is performed; cp. \cite{CiPi1,CiPi2,CiPi3}). The case when $l<n$ follows by H\"older's inequality and duality of Lorentz spaces (see Lemma~\ref{lem:conv}), but for $l=n$, one cannot argue in the same way and should use \eqref{eq:VS_k=n} for $\A=D^n$, as $\lebe^{1,1}=\lebe^1$ (of course, the inequality \eqref{eq:VS_k=n} in this case was known before, see, e.g., \cite{PVS}). Should one try to generalize this to linear elliptic operators of order $l$ on $\R^n$, i.e.,
\begin{align}\label{eq:pula}
\A u\in\lebe^{n/l,1}(\R^n,W)\implies u\in\hold_0(\R^n),
\end{align}
the question is not very interesting for $l<n$ since, by boundedness of singular integrals (again, see Lemma~\ref{lem:conv}), we have that \eqref{eq:pula} reduces to \eqref{eq:pizda}. For $l=n$, the characterization of elliptic $\A$ satisfying \eqref{eq:pula} is provided by Theorem~\ref{thm:main_infty}. We expand on this point in a functional framework in Section~\ref{sec:WA1}. In the future, we hope to be able to tackle generalizations to non--linear problems with $\lebe^1$--data (cp. \cite{MingioneKuusi,ACS}).

As it stands, condition~\eqref{eq:mistery_cond2} depends on the choice of Euclidean structure on the domain and target spaces. In Remark~\ref{rk:invariance}, we will show that \eqref{eq:mistery_cond2} is independent of this choice, as is ellipticity and the estimate \eqref{eq:VS_k=n}. We chose to present Theorem~\ref{thm:main_infty} in this form since, represented in coordinates, condition~\eqref{eq:mistery_cond2} is computable, whereas its invariant form in \eqref{eq:inv_cond} is implicit.

We conclude this section with the minor remark that Theorem~\ref{thm:main_infty} enables us to consistently add the case $n=1$ to \textsc{Van Schaftingen}'s theory \cite{VS}.
\subsection{Higher order operators}
We have the following $\mathscr{L}^n$--a.e.--generalization of \textsc{Dorronsoro}'s \cite[Thm.~1]{Dor}, concerning critical $\lebe^p$--differentiability of $\bv^k$--maps:
\begin{theorem}\label{thm:main_k_diff}
Let $\A$ be as in \eqref{eq:Ak}. Then:
\begin{enumerate}
\item\label{itm:a} If $1\leq j\leq \min\{k,n-1\}$, the following are equivalent:
\begin{enumerate}
\item[$(\mathrm{i})$] For all $u\in\bv^\A_{\locc}(\R^n)$, we have that
\begin{align*}
\D^{k-j}u\in \taylor^{j,n/(n-j)}(x)\quad\text{ for }\mathscr{L}^n\text{--a.e. }x\in\R^n.
\end{align*}
\item[$(\mathrm{ii})$] $\A$ is elliptic and canceling.
\end{enumerate}
\item\label{itm:b} If $k\geq n$, the following are equivalent:
\begin{enumerate}
\item[$(\mathrm{i})$] For all $u\in\bv^\A_{\locc}(\R^n)$, we have that
\begin{align*}
D^{k-n}u\in\taylor^{n,\infty}(x)\quad\text{ for }\mathscr{L}^n\text{--a.e. }x\in\R^n.
\end{align*}
\item[$(\mathrm{ii})$] $\A$ is elliptic and satisfies condition~\eqref{eq:mistery_cond2}.
\end{enumerate}
\end{enumerate}
In particular, if $k<n$ and $\A$ is elliptic and canceling, then every $u\in\bv^\A$ has a $k$--th order $\lebe^{n/(n-k)}$--Taylor expansion a.e.; if $k\geq n$ and $\A$ is elliptic and satisfies \eqref{eq:mistery_cond2}, then every $u\in\bv^\A$ is $k$ times differentiable a.e. (classically).
\end{theorem}
The spaces $\taylor^{k,p}(x)$ and $\Taylor^{k,p}(x)$ of $k$--times differentiable functions in the $\lebe^p$--sense at $x$ were introduced by \textsc{Calder\'on} and \textsc{Zygmund}. We recall their definitions in Section~\ref{sec:fspaces}. We also extend the definition of $u\in\bv^\A$ if $u\in\lebe^1$ and $\A u\in\mathcal{M}$. 

As it stands, Theorem~\ref{thm:main_k_diff} is a generalization of \cite[Thm.~1(i)]{Dor} only. In the future, we intend to refine the results of Theorem~\ref{thm:main_k_diff} to lower dimensional exceptional sets, i.e., in the spirit of \cite[Thm.~1(ii),(iii)]{Dor}.

The proof of Theorem~\ref{thm:main_k_diff} is a natural extension of the ideas used to prove Theorem~\ref{thm:main_k=1}. To prove~\ref{itm:a}, we again use \cite[Thm.~1.3]{VS}, as formulated in \eqref{eq:VS_j}. To prove~\ref{itm:b}, we use Theorem~\ref{thm:main_infty}.

The necessity of ellipticity of $\A$ for the conclusions of Theorem~\ref{thm:main_k_diff} is somewhat surprising, particularly since it is not necessary for the estimate \eqref{eq:VS_j} used in the proof of \ref{itm:a}, for $u\in\hold^\infty_c$, \emph{unless} $j=1$ (see \cite[Sec.~5.1]{VS} for detail). This is to say that ellipticity is not necessary to give critical estimates for lower derivatives, but, as we will see below in Theorem~\ref{thm:sub_crit}, ellipticity is necessary to get sub--critical $\lebe^p$--differentiability in our setting. 

With the machinery used to prove Theorem~\ref{thm:main_k_diff} in place, it is easy to prove the sub--critical version, which is characteristic of elliptic operators. We state it here for comparison with the critical case above and for completeness of this work.
\begin{theorem}\label{thm:sub_crit}
Let $\A$ be as in \eqref{eq:Ak}. Then:
\begin{enumerate}
\item\label{itm:asub} If $1\leq j\leq\min\{k,n-1\}$, $1<p<n/(n-j)$, $\A$ is elliptic if and only if for all $u\in\bv^\A_{\locc}$ we have
\begin{align*}
\D^{k-j}u\in\taylor^{j,p}(x)\quad\text{ for }\mathscr{L}^n\text{--a.e. }x\in\R^n.
\end{align*}
\item\label{itm:bsub} If $j=n\leq k$, $1<p<\infty$, $\A$ is elliptic if and only if for all $u\in\bv^\A_{\locc}$ we have 
\begin{align*}
\D^{k-j}u\in\taylor^{j,p}(x)\quad\text{ for }\mathscr{L}^n\text{--a.e. }x\in\R^n.
\end{align*}
\item\label{itm:csub} If $n<k$, $\A$ is elliptic if and only if for all $u\in\bv^\A_{\locc}$ we have 
\begin{align*}
D^{k-n-1}u\in\taylor^{n+1,\infty}(x)\quad\text{ for }\mathscr{L}^n\text{--a.e. }x\in\R^n.
\end{align*}
\end{enumerate}
In particular, in the third case, $u$ is $k$ times classically differentiable a.e..
\end{theorem}
The reader is invited to formulate the analogous  Lorentz variants of Theorems~\ref{thm:main_k_diff}, \ref{thm:sub_crit} in the spirit of Theorem~\ref{thm:main_k=1} or consult the Appendix (Section~\ref{sec:weak_est}).

We next connect our results to overdetermined elliptic linear systems with  $\lebe^1$--data. In \cite[Thm.~1]{CZ1}, \textsc{Calder\'on} and \textsc{Zygmund} showed that, for $1<p<\infty$, a $k$--th order linear elliptic system
\begin{align*}
\A u=f
\end{align*}
for which it is assumed that $u\in\sobo^{k,p}$, the pointwise regularity $f\in\taylor^{m,p}(x)$ can be improved to $\D^{k-j}u\in\taylor^{m+j,q}(x)$ for critical $q=np/(n-pj)$ for $1\leq j<n/p$. Their results rely on boundedness of singular integrals on $\lebe^p$ and cannot be extended to the limiting case $p=1$. In this case, by Ornstein's Non--inequality, it is not reasonable to assume that $u\in\sobo^{k,1}$, but we can extend the definition of $\taylor^{k,1}$ in a natural way, so that a Radon measure lies in $\taylor^{0,1}(x)$ for $\mathscr{L}^n$--a.e. $x$ by Lebesgue's Differentiation Theorem. If we then consider the system 
\begin{align*}
\A u=\mu,
\end{align*}
where $\mu$ is a Radon measure, Theorem~\ref{thm:sub_crit} implies that $\D^{k-j}u\in\taylor^{j,q}(x)$ for $\mathscr{L}^n$--a.e. $x$ for any sub--critical $q$ if and only if $\A$ is elliptic. In the range $1\leq k<n$, Theorem~\ref{thm:main_k_diff} implies that the critical exponent $q$ can be achieved if and only if, in addition, $\A$ is canceling. Finally, if $k\geq n$, condition~\eqref{eq:mistery_cond2} is equivalent with critical regularity. Of course, our results also imply that, if $\mu\in\bv^m_{\locc}$, we obtain $D^{k-j}u\in\taylor^{m+j,q}(x)$ for $\mathscr{L}^n$--a.e. $x$ with a similar discussion for the exponent $q$.
\newline\newline
\indent This paper is organized as follows: In Section~\ref{sec:prel_appdiff} we present general notation, the function spaces used, background facts about elliptic and EC operators, and the proof of sufficiency of ellipticity for Theorem~\ref{thm:sub_crit}. In Section~\ref{sec:proof_infty} we prove Theorem~\ref{thm:main_infty}, whereas its consequences for embeddings of function spaces are presented in Section~\ref{sec:WA1}. In Section~\ref{sec:char_mist_cond} we give an invariant, analytic characterization of condition~\eqref{eq:mistery_cond2} and examples to illustrate how it differs from cancellation. In 
Section~\ref{sec:k=1} we prove Theorem~\ref{thm:main_k=1}. In Section~\ref{sec:k} we prove Theorem~\ref{thm:main_k_diff}; this includes the necessity of ellipticity for Theorem~\ref{thm:sub_crit} (Lemma~\ref{lem:nec_ell_k}). In the Appendix (Section~\ref{sec:weak_est}), we present all differentiability results on the full Lorentz scale.
\section{Preliminaries}\label{sec:prel_appdiff}
Throughout this paper, we work on function spaces defined on $\R^n$, for $n\geq1$. 
We denote the ball centred at $x\in\R^n$ of radius $r>0$ by $\ball_r(x)\equiv\ball(x,r)$. We denote Lebesgue measure on $\R^n$ by $\mathscr{L}^n$ and the $s$--dimensional Hausdorff measure by $\mathscr{H}^s$. Averaged integrals, only taken with respect to Lebesgue measure ($\dif x$) throughout, are denoted by
\begin{align*}
\fint_{\Omega}\coloneqq\dfrac{1}{\mathscr{L}^n(\Omega)}\int_\Omega\eqqcolon(\,\,\bigcdot\,\,)_{\Omega}.
\end{align*}
The space of linear maps between two vector spaces $V,W$ is denoted by $\lin(V,W)$. The space $\mathcal{M}(\R^n,W)$ denotes the spaces of $W$--valued, bounded measures on $\R^n$ (i.e., of finite total variation), naturally equipped with the total variation norm $|\cdot|$. Inherently, $\mathcal{M}_{\locc}$ denotes the space of Radon measures. We denote by $\hold^\infty_c$ the space of continuously supported maps, by $\mathscr{S}$ the space of Schwartz functions, by $\lebe^p$ ($\lebe^p_{\locc}$) the space of (locally) $p$--integrable functions. We denote by $\mathscr{D}^\prime$ the space of distributions, the topological dual of $\hold^\infty_c$, and by $\mathscr{S}^\prime$ the space of tempered distributions, dual to the Schwartz space $\mathscr{S}$. Sobolev spaces are denoted by $\sobo^{k,p}$. The Fourier transform is defined by
\begin{align*}
\hat{u}(\xi)=\int_{\R^n}\e^{-\imag \xi\cdot x}u(x)\dif x\qquad\text{ for }\xi\in\R^n
\end{align*}
for $u\in\mathscr{S}$ and extended for tempered distributions by duality.

We write $X(\Omega,V)$ for space of (generalized) functions from some open $\Omega\subset\R^n$ into a vector space $V$. We abbreviate $X(\Omega)\coloneqq X(\Omega,\R)$, but, as a note of caution, we may suppress $V$ altogether, when no confusion arises (also for spaces of vector--valued functions). This is mostly done for displayed estimates.

We also make the rather strange but very handy convention to ignore multiplicative non--zero (possibly complex) constant scalars. Such constants play no role in neither the inequalities, nor in the algebraic conditions we discuss. This can be formalized by redefining equality as a suitable equivalence relation. For example, under this relation the Fourier transform of $\partial_j u$ at $\xi$ ``equals'' $\xi_j \hat{u}(\xi)$.
\subsection{Function spaces}\label{sec:fspaces}
In the spirit of \cite[Sec.~2.1]{GR}, we define for $\A$ as in \eqref{eq:Ak},
\begin{align*}
\bv^\A_{\locc}\coloneqq \bv^\A_{\locc}(\R^n)=\{u\in\lebe^1_{\locc}(\R^n,V)\colon\A u\in\mathcal{M}_{\locc}(\R^n,W)\},
\end{align*}
where $\mathcal{M}$ denotes the space of bounded measures. 
In general, we suppress the domain of maps defined locally or in full space. For each $u\in\bv^\A_{\locc}$, we denote the Radon--Nikod\'ym decomposition of $\A u\in\mathcal{M}_{\locc}$ by
\begin{align*}
\A u=\A^{ac}u\mathscr{L}^n+\A^su=\dfrac{\dif \A u}{\dif\mathscr{L}^n}\mathscr{L}^n+\A^su,
\end{align*}
where $\A^su\perp\mathscr{L}^n$.

We recall the spaces of $k$--times $\lebe^p$--differentiable functions $\Taylor^{k,p}(x)$, $\taylor^{k,p}(x)$, as introduced for $x\in\R^n$, $1\leq p\leq\infty$ in \cite[Def.~1--2]{CZ1} as the spaces of maps $u$ defined in a neighbourhood of $x$ for which there exists a polynomial $P\eqqcolon P^k_xu$ of degree at most $k$ (a \emph{$k$--th order $\lebe^p$--Taylor expansion}) such that
\begin{align}\label{eq:def_taylor_p}
\left(\fint_{\ball_r(x)}|u-P|^p\dif y\right)^{1/p}=O(r^k)\text{, respectively }o(r^k),\quad\text{ as }r\downarrow0.
\end{align}
If $p=\infty$, we replace the LHS by $\|u-P\|_{\lebe^\infty(\ball_r(x))}$. We denote the \emph{remainder} by $R^k_x u\coloneqq u-P^k_xu$. It is easy to see that $\Taylor^{k,q}(x)\subset\Taylor^{k,p}(x)$ for $1\leq p\leq q\leq\infty$; the same holds true for $\taylor^{k,p}$. If $u\in\taylor^{k,\infty}(x)$, then $u$ has $k$ classical derivatives at $x$. For more detail on these spaces, see \textsc{Ziemer}'s monograph \cite[Ch.~3]{Z}.

The Lorentz space $\lebe^{p,q}$ consists, for $1\leq p<\infty$, $1\leq q\leq\infty$, of measurable functions $u$ such that the quasi--norm
\begin{align*}
&\|u\|_{\lebe^{p,q}}\coloneqq p^{1/q}\left(\int_0^\infty \left[\lambda\mathscr{L}^n\left(\{|u|>\lambda\}\right)^{1/p}\right]^q\frac{\dif \lambda}{\lambda}\right)^{1/q}\quad\text{ if }q<\infty\\
&\|u\|_{\lebe^{p,\infty}}\coloneqq \sup_{\lambda>0}\lambda\mathscr{L}^n\left(\{|u|>\lambda\}\right)^{1/p}
\end{align*}
is finite. Lorentz spaces are a refinement of the Lebesgue scale, as can be seen, for example, since $\lebe^{p,p}=\lebe^p$ and $\lebe^{p,\infty}=\lebe^p_{\text{weak}}$. More precisely, $\lebe^{p,q_1}\subset\lebe^{p,q_2}$ if $q_1\leq q_2$, and $\lebe^{p_1,\infty}_{\locc}\subset\lebe^{p_2,1}_{\locc}$ if $p_2<p_1$. 
\subsection{Elliptic operators}
We begin with an important Green--type formula:
\begin{lemma}[{\cite[Lem.~2.1]{BVS}}]\label{lem:conv}
Let $\A$ as in \eqref{eq:Ak} be elliptic. Then there exists a map $K^\A\in\hold^\infty(\R^n\setminus\{0\},\lin(W,V))\cap\lebe^1_{\locc}(\R^n,\lin(W,V))$ such that
\begin{align}\label{eq:conv_rep}
u(x)=\int_{\R^n}K^\A(x-y)\A u(y)\dif y=(K^\A\star\A u)(x)
\end{align}
for all $u\in\hold^\infty_c(\R^n,V)$ and $x\in\R^n$. Moreover, for all integers $l\geq\max\{0,k-n+1\}$, the map $\D^lK^\A$ is $(k-n-l)$--homogeneous.
\end{lemma}

We next show that under the mild assumption of ellipticity, maps in $\bv^\A_{\locc}$ have appropriate local Sobolev regularity for their lower derivatives. The proof is reminiscent of the {Deny--Lions} Lemma, which covers the case $\A=\D$ \cite{DL} (see also \cite[Thm.~2.1]{ST} for the case $\A=\mathcal{E}$). Our proof is based on \cite[Thm.~4.5.8]{HGOD1}.
\begin{lemma}\label{lem:ell=>sob}
Let $\A$ as in \eqref{eq:Ak} be elliptic and $u\in\mathscr{D}^\prime(\R^n,V)$ be such that $\A u\in\mathcal{M}_{\locc}(\R^n,W)$. Then $u\in\sobo^{k-1,p}_{\locc}$ for $1\leq p<n/(n-1)$.
\end{lemma}
\begin{proof}
It is easy to see that $\A^*\A$ is elliptic. We write $E\coloneqq K^{\A^*\A}\in\hold^\infty(\R^n\setminus\{0\},\lin(V,V))$, as given by Lemma \ref{lem:conv}. 
Choose $\rho\in\hold^\infty_c(\R^n)$ such that $\rho=1$ in an open cube $Q\subset\R^n$. It is then the case that
\begin{align*}
\rho u=E\star\A^*\A (\rho u)=(\A E^*)^*\star \A(\rho u),
\end{align*}
where the latter inequality follows by integration by parts. It follows that
\begin{align*}
\D^{k-1}(\rho u)&=(\A E^*)^*\star \A(\rho u)\\
&=D^{k-1}(\A E^*)^*\star (\rho\A u)+\sum_{j=1}^k D^{k-1}(\A E^*)^*\star B_j(D^{j}\rho,D^{k-j}u)\\
&\eqqcolon\mathbf{I}+\sum_{j=1}^k\mathbf{II}_j
\end{align*}
in the sense of distributions. Here $B_j$ are bilinear pairings that depend on $\A$ only, as given by the Leibniz rule. If $n>1$, by Lemma \ref{lem:conv}, we have that $D^{2k-1}E^*$ is $(1-n)$--homogeneous, hence so is $D^{k-1}(\A E^*)^*$ by $k$--homogeneity of $\A$. In particular, since $E$ is smooth it follows that $|D^{k-1}(\A E^*)^*|\lesssim |\cdot|^{1-n}$, so $\mathbf{I}\in\lebe^p_{\locc}$ for $1\leq p<n/(n-1)$ by boundedness of Riesz potentials \cite[Thm.~V.1]{Stein} and inclusions of Lorentz spaces. If $n=1$, one can infer from \eqref{eq:conv_ker} that $|D^{k-1}(\A E^*)^*|\lesssim 1+\log|\cdot|$, so it is still the case the case that $D^{k-1}(\A E^*)^*\in\lebe^p_{\locc}$ for any $1\leq p<\infty$, hence $\mathbf{I}\in\lebe^p_{\locc}$ by Young's convolution inequality. To conclude, we claim that $\mathbf{II}_j$ are smooth in $Q$. This follows from an application of \cite[Thm.~4.2.5]{HGOD1}, which implies that
\begin{align*}
\singspt (f\star g)&\subset \singspt f+\singspt g\\
&\subset \{0\}+\R^n\setminus Q=\R^n\setminus Q,
\end{align*}
where we wrote $f=D^{k-1}(\A E^*)^*$, $g=B_j(D^j\rho,\D^{k-j}u)$ (the singular support $\singspt f$ of a distribution $f$ is defined in \cite[Def.~2.2.3]{HGOD1} as the set of points $x$ such that no restriction of $f$ to a neighbourhood of $x$ equals a smooth map). It follows that $D^{k-1}u\in\lebe^p_{\locc}(Q)$ for $1\leq p<n/(n-1)$, so $\D^{k-1}u\in\lebe^{p}_{\locc}(\R^n)$ as $Q$ is arbitrary. 

The conclusion follows either by iterating the above argument for $\A=D^{k-j}$, $j=1,\ldots, k$, or by standard theory for Sobolev Spaces \cite[Sec.~1.1.11]{Mazya}.
\end{proof}
The following Lemma is essentially due to \textsc{Calder\'on} and \textsc{Zygmund} and covers the non--trivial statement that if a higher derivative of a Sobolev map $u$ has $\lebe^p$--derivatives at $x$, then the lower derivatives of $u$ have more $\lebe^p$--derivatives at $x$. More precisely, we show exchangeability of weak and $\lebe^p$--derivatives
\begin{align}\label{eq:commutativity}
\nabla(D^{k-1} u)=\nabla^j(D^{k-j} u)\qquad\text{ $\mathscr{L}^n$--a.e.}
\end{align}
for $j=2,\ldots,k$, which is trivial if $u\in\sobo^{k,p}_{\locc}$, but less so if we only assume that $u\in\sobo^{k-1,p}_{\locc}$ and that the LHS of \eqref{eq:commutativity} is well defined. Here $\nabla^j u(x)=D^jP^j_xu(x)$.
\begin{lemma}\label{lem:CZ_lemma}
Let $1\leq p<n$, $u\in\sobo^{k-1,p}_{\locc}$, $x\in\R^n$ be such that $\D^{k-1}u\in\taylor^{1,p}(x)$. Then $\D^{k-j}u\in\taylor^{j,p}(x)$ for $j=1\ldots k$.
\end{lemma}
Note that due to Ornstein's Non--inequality, it is crucial that we do not assume that $u\in\sobo^{k,p}_{\locc}$, as we aim to apply this to $u\in\bv^\A_{\locc}$ (cp. Lemma~\ref{lem:ell=>sob}).
\begin{proof}
We argue by induction on $j$, using the fact following from \cite[Thm.~11.1]{CZ1}, namely that if $f\in\sobo^{1,p}_{\locc}$ such that $\D f\in\taylor^{l,p}(x)$, then $f\in\taylor^{l+1,p}$. The statement is true for $j=1$. We apply the result mentioned above for $f=\D^{k-j-1}u$ with $l=j=1,\ldots, k-1$ to get the conclusion.
\end{proof}
The next Lemma is a consequence of the main result in \cite{ABC} and follows from the first order case \cite[Lem.~3.1]{GR2} by replacing $u$ with $\D^{k-1}u$. As the result is crucial for Theorems~\ref{thm:main_k=1}, \ref{thm:main_k_diff}, \ref{thm:sub_crit}, we include a complete proof.
\begin{lemma}\label{lem:algebra_Au} Let $\A$ as in \eqref{eq:Ak} be elliptic, $u\in\bv^\A_{\locc}$. Then for $\mathscr{L}^n$--a.e. $x\in\R^n$, we have that $\D^{k-1}u$ is $\lebe^p$--differentiable at $x$, with (first) approximate gradient $\nabla^k u(x)\coloneqq \nabla(D^{k-1}u)(x)\in V\odot^k\R^n$ such that
\begin{align}\label{eq:algebra_Au}
\A^{ac}u(x)\coloneqq \dfrac{\dif\A u}{\dif\mathscr{L}^n}(x)=A(\nabla^ku(x)),
\end{align}
where $1\leq p<n/(n-1)$.
\end{lemma}
\begin{proof}
There is no loss of generality in assuming that $u\in\bv^\A(\R^n)$, otherwise multiply $u$ with a smooth cut--off function $\rho$ that equals $1$ in an arbitrarily large set. One then uses the Leibniz rule and Lemma~\ref{lem:ell=>sob} to show that $\A(\rho u)\in\mathcal{M}(\R^n,W)$. By Lemma~\ref{lem:conv} and a standard regularization argument, we have that $\D^{k-1}u=D^{k-1}K^\A\star \A u$, where $D^{k-1}K^\A$ is $(1-n)$--homogeneous. The $\lebe^p$--differentiability statement follows from \cite[Thm.~3.4]{ABC}.

It remains to prove \eqref{eq:algebra_Au}. Let $x\in\R^n$ be a Lebesgue point of $D^{k-1}u$ and $\A^{ac}u$, and also a point of $\lebe^1$--differentiability of $u$. We also consider a sequence $(\eta_\varepsilon)_{\varepsilon>0}$ of standard mollifiers, i.e., $\eta_1\in\hold^\infty_c(\ball_1(0))$ is radially symmetric and has integral equal to 1 and $\eta_\varepsilon(y)=\varepsilon^{-n}\eta_1(x/\varepsilon)$. Finally, we denote $u_\varepsilon:=u\star\eta_\varepsilon$ and compute
\begin{align*}
\D^k u_\varepsilon(x)&=\int_{\ball_\varepsilon(x)} D^{k-1}u(y)\otimes\D_x\eta_\varepsilon(x-y)\dif y\\
&=-\int_{\ball_\varepsilon(x)}(P^1_xD^{k-1}u)(y)\otimes\D_y\eta_\varepsilon(y-x)\dif y\\
&\quad+\int_{\ball_\varepsilon(x)} (R^1_xD^{k-1}u)(y)\otimes\D_x\eta_\varepsilon(x-y)\dif y\\
&=\int_{\ball_\varepsilon(x)}\eta_\varepsilon(y-x)\nabla(D^{k-1} u)(x)\dif y+\\
&\quad+\int_{\ball_\varepsilon(x)} (R^1_xD^{k-1}u)(y)\otimes\D_x\eta_\varepsilon(x-y)\dif y\\
&=\nabla^k u(x)+\int_{\ball_\varepsilon(x)}(R^1_xD^{k-1}u)(y)\otimes\D_x\eta_\varepsilon(x-y)\dif y,
\end{align*}
where we used integration by parts to establish the third equality. Since \begin{align*}
\|\D_x\eta(x-\cdot)\|_\infty=\varepsilon^{-(n+1)}\|\D\eta_1\|_\infty,
\end{align*}
we have that $|\D^k u_\varepsilon(x)-\nabla^k u(x)|\leq c(n,\eta_1)\varepsilon^{-1}(|R_xD^{k-1}u|)_{x,\varepsilon}=o(1)$ as $\varepsilon\downarrow0$ since $x$ is a point of $\lebe^1$--differentiability of $D^{k-1}u$. In particular, $\D^k u_\varepsilon\rightarrow\nabla^k u$ $\mathscr{L}^n$--a.e., so that $\A u_\varepsilon\rightarrow A(\nabla^k u)$ $\mathscr{L}^n$--a.e. To conclude, we will also show that $\A u_\varepsilon\rightarrow\A^{ac}u$ $\mathscr{L}^n$--a.e. We have that $\A u_\varepsilon=\A u\star\eta_\varepsilon$, so 
\begin{align*}
\A u_\varepsilon(x)-\A^{ac}u(x)&=\A^{ac}u\star\eta_\varepsilon(x)-\A^{ac}u(x)+\A^s u\star\eta_\varepsilon(x)\\
&=\int_{\ball_\varepsilon(x)}\eta_\varepsilon(x-y)\left(\A^{ac}u(y)-\A^{ac}u(x)\right)\dif y \\
&\quad+\int_{\ball_\varepsilon(x)}\eta_\varepsilon(x-y)\dif \A^s u(y).
\end{align*}
Using the facts that $\|\eta_\varepsilon(x-\cdot)\|_\infty=\varepsilon^{-n}\|\eta_1\|_\infty$ and that $x$ is a Lebesgue point of $\A^{ac}u$, we conclude that both integrals converge to zero as $\varepsilon\downarrow0$.
\end{proof}
\begin{proof}[Proof of sufficiency of ellipticity for Theorem~\ref{thm:sub_crit}]
Assume that $\A$ is elliptic and let $u\in\bv^\A_{\locc}$. From Lemma~\ref{lem:algebra_Au}, we have that for $\mathscr	{L}^n$--a.e. $x\in\R^n$, $\D^{k-1}u\in\taylor^ {1,p}(x)$ for all $1\leq p<n/(n-1)$. Apply \cite[Thm.~11.1]{CZ1} inductively until~\ref{itm:asub} is proved. If $k\geq n$, apply \cite[Thm.~11.1]{CZ1} once more to get \ref{itm:bsub}. If $k>n$, apply \cite[Thm.~11.2]{CZ1} to \ref{itm:bsub} once to get \ref{itm:csub}.
\end{proof}
\subsection{EC operators}\label{sec:EC}
We give an analytic characterization of canceling 
operators: 
\begin{lemma}\label{lem:canc_char}
Let $\A$ as in \eqref{eq:Ak} be elliptic. Then $\A$ is canceling if and only if whenever the equation
\begin{align}\label{eq:canc_eq}
\A u=\delta_0w,
\end{align}
has a solution $u\in\bv^\A_{\locc}$ for some $w\in W$, we necessarily have $w=0$. Moreover, if $\A$ is elliptic and non--canceling, \eqref{eq:canc_eq} has a solution $u_h\in\hold^\infty(\R^n\setminus\{0\},V)$ such that $\D^lu_h$ is $(k-n-l)$--homogeneous for all $l\geq\max\{0,k-n+1\}$.
\end{lemma}
\begin{proof}
Suppose that $\A$ is EC. By ellipticity, $\A$ has an exact annihilator (analogous to $\curl$ for $\A=\D$), i.e., there exists a homogeneous differential operator $\mathcal{A}$ such that $\ker\mathcal{A}[\xi]=\mathrm{im\,}\A[\xi]$ for all $\xi\neq0$. An example would be the projection operator defined by 
\begin{align*}
\mathcal{A}[\xi]\coloneqq \det(\Delta_\A[\xi])\left(\text{Id}-\A[\xi](\Delta_\A[\xi])^{-1}\A^*[\xi]\right),
\end{align*}
where $\Delta_\A\coloneqq \A^*\A$ (see \cite[Rk.~4.1, Sec.~4.2]{VS}). Applying $\mathcal{A}$ to \eqref{eq:canc_eq}, we get that $\mathcal{A}(\delta_0w)=0$. We apply the Fourier transform to get that $\mathcal{A}[\xi]w=0$ for each non--zero $\xi$. In particular, 
\begin{align*}
w\in\bigcap_{\xi\in \mathbb{S}^{n-1}}\ker\mathcal{A}[\xi]=\bigcap_{\xi\in \mathbb{S}^{n-1}}\mathrm{im\,}\A[\xi],
\end{align*}
so that $w=0$.

Conversely, suppose that $\A$ is elliptic, non--canceling so there exists a non--zero $w\in\mathrm{im\,}\A[\xi]$ for any non--zero $\xi$. We define the map $u_h$ via 
\begin{align*}
u_h\coloneqq K^\A w,
\end{align*}
where, $K^\A$ is as in Lemma~\ref{lem:conv}. Recall from the proof of \cite[Lem.~2.1]{BVS} that $K^\A$ is, in addition, a tempered distribution such that 
\begin{align*}
\widehat{K^\A}(\xi)=\A^\dagger[\xi]\quad\text{ for }\xi\neq0,
\end{align*}
where $\A^\dagger[\xi]=(\Delta_\A[\xi])^{-1}\A^*[\xi]$. To be precise, $\widehat{K^\A}=(\A^\dagger[\cdot])^{\bigcdot}$ in the sense of tempered distributions, where $f\mapsto f^{\bigcdot}$ denotes the extension of a homogeneous distribution $f\in(\mathscr{D}^\prime\cap \hold^{\infty})(\R^n\setminus\{0\})$, as defined in \cite[Thm.~3.2.3-4]{HGOD1}. In this case, $f^{\bigcdot}$ is a tempered distribution by \cite[Thm.~7.1.18]{HGOD1}.

It is now clear that the smoothness and homogeneity properties of $u_h$ follow from those of $K^\A$. The $\bv^\A_{\locc}$--regularity follows from Lemma~\ref{lem:ell=>sob}, as soon as we prove that $\A u_h=\delta_0 w$. This follows from the fact that 
\begin{align*}
\widehat{\A u_h}(\xi)=\A[\xi]\widehat{u_h}(\xi)=\A[\xi]\A^\dagger[\xi]w=w,\quad\text{ for }\xi\neq0,
\end{align*}
where the last inequality follows from the definition of $w$ and the elementary fact that $\A[\xi]\A^\dagger[\xi]$ is the orthogonal projection onto $\mathrm{im\,}\A[\xi]$. This implies by \cite[(3.2.26)]{HGOD1} that $\A[\cdot]\widehat{u_h}^{\bigcdot}=w^{\bigcdot}$. Fourier inverting the previous equality gives \eqref{eq:canc_eq}.
\end{proof}
In fact, while proving Lemma~\ref{lem:canc_char}, we have proved the independently interesting fact that, for an elliptic operator $\A$ as in \eqref{eq:Ak}, we have
\begin{align*}
\bigcap_{\xi\in\mathbb{S}^{n-1}}\mathrm{im\,}\A[\xi]=\{w\in W\colon \A u=\delta_0 w\text{ for some }u\in\bv^\A_{\locc}\}.
\end{align*}
We will not make explicit use of this fact, but we will outline some consequences:
\begin{enumerate}
\item If $\A$ is determined ($\dim V=\dim W$), then $\A$ has a genuine fundamental solution $E$, by which we mean $\A E=\id_V \delta_0$ (cp. proof of Lemma~\ref{lem:ell=>sob}).
\item If $\A$ is overdetermined ($\dim V<\dim W$), assuming that the latter is minimal), no such $E$ exists. The kernel $K^\A$ from Lemma~\ref{lem:conv} is a significantly less flexible replacement.
\end{enumerate}
This highlights a sharp difference between both the scalar and vectorial case (recall \cite{Ehrenpreis,Malgrange}), as well as between the determined and overdetermined case.
\section{Proof of Theorem~\ref{thm:main_infty}}\label{sec:proof_infty}
We begin with a formal algebraic computation in Fourier Space, which can be made precise (equality in the sense of tempered distributions) by the arguments in the proof of \cite[Lem.~2.1]{BVS}. For $u$ in $\hold^\infty_c(\R^n,V)$ and $\xi\in\R^n$, we have that
\begin{align*}
\A[\xi]\hat{u}(\xi)&=\widehat{\A u}(\xi)\\
\hat{u}(\xi)&=\A^\dagger[\xi]\widehat{\A u}(\xi)\\
\widehat{D^{k-n}u}(\xi)&=L[\xi]\widehat{\A u}(\xi)\\
D^{k-n}u&=\check{L}\star\A u,
\end{align*}
where $L[\xi]\in\lin(W,V\odot^{k-n}\R^n)$ is given by
\begin{align*}
L[\xi]w\coloneqq \A^\dagger[\xi]w\otimes^{k-n}\xi
\end{align*}
for $\xi\in\R^n{\setminus}\{0\}$, $w\in W$. We can also assert that $\check{L}\in\hold^\infty(\R^n\setminus\{0\})\cap\lebe^1_{\locc}(\R^n)$, valued in $\lin(W,V\odot^{n-k}\R^n)$.

Essentially by \cite[Thm.~7.1.20]{HGOD1}, we can write
\begin{align}\label{eq:conv_ker}
\check{L}=H_0+\log|\cdot|\int_{\mathbb{S}^{n-1}}L[\xi]\dif\mathcal{H}^{n-1}(\xi)=H_0+\log|\cdot|\mathcal{L},
\end{align}
where $H_0\in\hold^\infty(\R^n\setminus\{0\},\lin(W,V\odot^{n-k}\R^n))$ is zero--homogeneous. In particular, $H_0$ is essentially bounded in $\R^n$. We prove Formula \eqref{eq:conv_ker} at the end of this section.

\begin{proof}[Proof of necessity] 
	Let $w$ fail condition~\eqref{eq:mistery_cond2}. Then $w\neq0$, so there exists $u_h\in\lebe^1_{\locc}(\R^n,V)$ such that $\A u_h=\delta_0 w$ by the proof of Lemma~\ref{lem:canc_char}. By \eqref{eq:conv_ker},
\begin{align*}
\|D^{k-n}u_h\|_{\lebe^\infty}\geq|\|H_0w\|_{\lebe^\infty}-\|\log|\cdot|\mathcal{L}w\|_{\lebe^{\infty}}|,
\end{align*}
which is clearly infinite (near 0) since $\mathcal{L}w\neq0$ and $H_0$ is bounded.
\end{proof}
We next show that condition~\eqref{eq:mistery_cond2} is sufficient for the estimate \eqref{eq:VS_k=n}. The ideas we use originate in \cite[Sec.~2]{BVS}. By \eqref{eq:conv_ker}, the triangle inequality, and Young's Convolution Inequality, we have that
\begin{align*}
\|D^{k-n}u\|_{\lebe^\infty}&\lesssim\|H_0\star\A u\|_{\lebe^\infty}+\|\log|\cdot|\star[\mathcal{L}\A u]\|_{\lebe^\infty}\\
&\leq\|H_0\|_{\lebe^\infty}\|\A u\|_{\lebe^1}+\|\log|\cdot|\star[\mathcal{L}\A u]\|_{\lebe^\infty},
\end{align*}
so it suffices to prove that 
\begin{align*}
\|\log|\cdot|\star[\mathcal{L}\A u]\|_{\lebe^\infty}\lesssim\|\A u\|_{\lebe^1},
\end{align*}
for $u\in\hold^\infty_c(\R^n,V)$. Equivalently, we want to show that for all $v\in V$, $\eta\in\R^n$ of unit length, we have that
\begin{align}\label{eq:log_est}
\left|\int_{\R^n}\langle\log|y|v\otimes^{k-n}\eta,\mathcal{L}\A u(x-y)\rangle\dif y\right|\lesssim\|\A u\|_{\lebe^1}\quad\text{for }\mathscr{L}^n\text{--a.e. }x\in\R^n.
\end{align}
The proof of \eqref{eq:log_est} will follow quite easily from the following Lemma:
\begin{lemma}\label{lem:BVS_var}
Let $\A$ be elliptic and satisfying condition~\eqref{eq:mistery_cond2}. Then there exists an integer $l$ such that for all $u\in\hold^\infty_c(\R^n,V)$ and all $\varphi\in\hold^\infty(\R^n\setminus\{0\},\mathrm{im\,}\mathcal{L}^*)$ such that $|\cdot|^j|D^j\varphi|\in\lebe^1_{\locc}(\R^n)$ for $j=0,1,\ldots, l$ we have that
\begin{align*}
\left|\int_{\R^n}\langle\varphi,\A u\rangle\dif x\right|\lesssim\sum_{j=1}^l\int_{\R^n}|\A u||\cdot|^j|D^j\varphi|\dif x.
\end{align*}
\end{lemma}
Lemma~\ref{lem:BVS_var} ammounts to a minor, but crucial algebraic modification of \cite[Lem.~2.2]{BVS} (cp. \cite[Prop.~8.9]{VS}; see also \cite[Lem.~2.5]{VS}), ideas of which originate in \cite{BB07,VS_BMO} (cp. \cite{Mazya_JEMS,BousquetMironescu}).
\begin{proof}[Proof of Lemma~\ref{lem:BVS_var}]
We denote by $\mathcal{A}=\sum_{|\alpha|=l}\partial^\alpha\mathcal{A}_\alpha$ an exact annihilator of $\A $, i.e.,
$
\ker\mathcal{A}[\xi]=\mathrm{im\,}\A[\xi]
$
for all $\xi\in\R^n\setminus\{0\}$. Such an operator exists by \cite[Prop.~4.2]{VS}.

Since $(\xi^\alpha)_{|\alpha|=l}$ is a basis for homogeneous polynomials of degree $l$, we have that $w_0\in\ker\mathcal{A}[\xi]$ for all $\xi\neq0$ is equivalent with $w_0$ lying in the kernel of the map $T\colon w\mapsto(\mathcal{A}_\alpha w)_{|\alpha|=l}$. By condition~\eqref{eq:mistery_cond2}, we have that $\mathrm{im\,}\mathcal{L}^*\cap\bigcap_{\xi\in\mathbb{S}^{n-1}}\ker\mathcal{A}[\xi]=\{0\}$, hence the restriction of $T$ to $\mathrm{im\,}\mathcal{L}^*$ is injective. Equivalently, this restriction is left--invertible, so there exist linear maps $K_\alpha\in\lin(W,\mathrm{im\,}\mathcal{L}^*)$ such that
\begin{align*}
\sum_{|\alpha|=l}K_\alpha \mathcal{A}_\alpha\restriction_{\mathrm{im\,}\mathcal{L}^*}=\id_{\mathrm{im\,}\mathcal{L}^*}.
\end{align*}
The remainder of the proof follows exactly as in \cite[pp.1426]{BVS}. We reproduce the argument for the convenience of the reader.

Define the matrix--valued field
\begin{align*}
P(x)\coloneqq\sum_{|\alpha|=l}\dfrac{x^\alpha}{\alpha!}K_\alpha^*,
\end{align*}
which is essentially a right--inverse (integral) of $\mathcal{A}^*$, as
\begin{align}\label{eq:right_FS}
\mathcal{A}^*P=\sum_{|\alpha|=l}\mathcal{A}^*_\alpha\partial^\alpha P=\sum_{|\alpha|=l}\mathcal{A}^*_\alpha K^*_{\alpha}=\id_{\mathrm{im\,}\mathcal{L}^*}.
\end{align}
We next claim that the following integration by parts formula holds
\begin{align}\label{eq:ibp}
	0=\int_{\R^n}\langle P\varphi,\mathcal{A}(\A u)\rangle\dif x=(-1)^l\int_{\R^n}\langle \mathcal{A}^*[P\varphi],\A u\rangle\dif x,
\end{align}
where the first equality follows simply by $\mathcal{A} \circ\A\equiv0$. To prove the second equality, we consider cut--off functions $\rho_r\in\hold^\infty_c(B_{2r}(0),[0,1])$ such that $\rho=1$ in $B_r(0)$ and $|D^j\rho_r|\lesssim r^{-j}$ for $j=0,\ldots, l$. It is clear that $D^j\rho_r\rightarrow 0$ $\mathscr{L}^n$--almost everywhere as $r\downarrow0$ and $|D^j\rho_r|\lesssim|\cdot|^{-j}$ for $j=0,\ldots l$. By the dominated convergence theorem and integration by parts for smooth maps, we have that
\begin{align*}
	\int_{\R^n}\langle P\varphi,\mathcal{A}(\A u)\rangle\dif x&=\lim_{r\downarrow0}\int_{\R^n}\langle (1-\rho_r) P\varphi,\mathcal{A}(\A u)\rangle\dif x\\
	&=(-1)^l\lim_{r\downarrow0}\bigg(\int_{\R^n}\langle (1-\rho_r)\mathcal{A}^*[P\varphi],\A u\rangle\dif x+\bigg.\\
	&\bigg.+\sum_{j=1}^lB_j(D^j\rho_r,D^{l-j}[P\varphi])\dif x\bigg),
\end{align*} 
where $B_j$ are bilinear pairings on finite dimensional spaces that depend on $\mathcal{A} $ only.  By the Leibniz rule and the assumption on the singularity of  $\varphi$ at zero, we have that $|\cdot|^{j}D^{l-j}[P\varphi]\in\lebe^1_{\locc}$ for $j=0,\ldots,l$. This enables us to conclude by the dominated convergence theorem applied to each term above that \eqref{eq:ibp} holds.

By definition of $\varphi$, \eqref{eq:right_FS}, and \eqref{eq:ibp}, we have that
\begin{align*}
\left|\int_{\R^n}\langle\varphi,\A u\rangle\dif x\right|&=\left|\int_{\R^n}\langle[\mathcal{A}^*P]\varphi,\A u\rangle\dif x\right|\\
&=\left|\int_{\R^n}\langle[\mathcal{A}^*P]\varphi-\mathcal{A}^*[P\varphi],\A u\rangle\dif x\right|.
\end{align*}
The presence of the ``associator'' ensures the elimination of the essentially inhomogeneous zero--th order term $\varphi$:
\begin{align*}
[\mathcal{A}^*P]\varphi-\mathcal{A}^*[P\varphi]=[\mathcal{A}^*P]\varphi-[\mathcal{A}^*P]\varphi-\sum_{j=1}^l\tilde B_j(D^j\varphi,D^{l-j}P),
\end{align*}
where $\tilde B_j$ is another set of bi--linear pairings arising also from the product rule. By $l$--homogeneity of $P$, the conclusion follows.
\end{proof}

\begin{proof}[Proof of sufficiency] To prove \eqref{eq:log_est}, we need only apply Lemma~\ref{lem:BVS_var} to the map 
\begin{align*}
\varphi\coloneqq\log|\cdot|\mathcal{L}^*(v\otimes^{k-n}\eta),
\end{align*}
which has suitably homogeneous derivatives, to get that
\begin{align*}
\mathrm{LHS}\eqref{eq:log_est}&\lesssim\sum_{j=1}^l\int_{\R^n}|\A u(x-y)||y|^j|D^j\log|y|\|\mathcal{L}^*\|\dif y\\
&\lesssim\|\mathcal{L}^*\|\|\A u\|_{\lebe^1}.
\end{align*}
The proof is complete.
\end{proof}
\begin{proof}[Proof of Formula \eqref{eq:conv_ker}] It is shown in \cite[Thm.~3.2.4,~7.1.18]{HGOD1} that a $(-n)$--homo-geneous map $f\in\hold^\infty(\R^n\setminus\{0\})$ can be extended to a tempered distribution $f^{\bigcdot}$ satisfying a weakened homogeneity property \cite[(3.2.24)$^\prime$]{HGOD1}. The (inverse) Fourier Transform of $f^{\bigcdot}$ then equals $h+\log|\cdot|\int_{\mathbb{S}^{n-1}}f\dif\mathscr{H}^{n-1}$ by \cite[(7.1.19)]{HGOD1}, where $h$ is a $0$--homogeneous map in $\hold^\infty(\R^n\setminus\{0\})$. We apply this to $L$ component wise.
\end{proof}
\section{Remarks on and consequences of Theorem~\ref{thm:main_infty}}\label{sec:rk_Linfty}
\subsection{Embeddings of $\sobo^{\A,1}$ and $\bv^\A$}\label{sec:WA1}
Throughout this Section, we restrict our attention to operators $\A$ on $\R^n$ of order $k=n$ and compare our result with the embeddings $\sobo^{n,1}(\R^n)\hookrightarrow\hold_0$ (the space of continuous functions that vanish at infinity) and $\bv^n(\R^n)\hookrightarrow\lebe^\infty$. To this end, we consider the space
\begin{align*}
\sobo^{\A,1}(\R^n)\coloneqq\{u\in\lebe^1(\R^n,V)\colon\A u\in\lebe^1(\R^n,W)\},
\end{align*}
which is a Banach space when endowed with the obvious norm $\|u\|_{\sobo^{\A,1}}\coloneqq\|u\|_{\lebe^1}+\|\A u\|_{\lebe^1}$ (see also \cite{BDG,GR}). We have the following:
\begin{theorem}\label{thm:sobo}
Let $\A$ be as in \eqref{eq:Ak} be elliptic of order $k= n\geq 1$. The following are equivalent:
\begin{enumerate}
\item $\A$ satisfies \eqref{eq:mistery_cond2}.
\item\label{it:sobo} $\sobo^{\A,1}(\R^n)\hookrightarrow \hold_0(\R^n,V)$.
\item\label{it:bv} $\bv^{\A}(\R^n)\subset \lebe^{\infty}(\R^n,V)$ with $\|u\|_{\lebe^\infty}\leq c|\A u|(\R^n)$ for $u\in\bv^\A(\R^n)$.
\end{enumerate}
\end{theorem}
\begin{proof}
Necessity of \eqref{eq:mistery_cond2} follows directly from Theorem~\ref{thm:main_infty}. Assume now that $\A$ satisfies \eqref{eq:mistery_cond2}, so Theorem~\ref{thm:main_infty} implies that
\begin{align}\label{eq:k=n}
\|u\|_{\lebe^\infty}\leq c\|\A u\|_{\lebe^1}
\end{align}
for all $u\in\hold^\infty_c(\R^n,V)$. Since test functions are norm--dense in $\sobo^{\A,1}(\R^n)$ \cite[Thm.~2.8]{BDG}, we have that $\sobo^{\A,1}(\R^n)$ embeds in the uniform closure of $\hold^\infty_c$, which is $\hold_0$, so \ref{it:sobo} is proved.

In the case of $\bv^\A$, test functions are not norm--, but strictly--dense (more precisely, $\A$--strictly dense \cite[Thm.~ 2.8]{BDG}) and it is easy to see that addition is not continuous in the strict topology on the space of bounded measures. In particular, we cannot prove that $\bv^\A$ embeds in $\hold_0$ (this is clearly visible if one looks at the indicator function of $(0,1)$ in dimensions $n=1$). Instead, one can prove \ref{it:bv}: Let $u_j\in\hold^\infty_c(\R^n)$ be such that $u_j\rightarrow u$ $\mathscr{L}^n$--a.e. and $|\A u_j|(\R^n)\rightarrow|\A u|(\R^n)$ (mollifications of $u$ satisfy this). 
It then follows that
\begin{align*}
\|u\|_{\lebe^\infty}\leq \liminf_{j\rightarrow\infty}\|u_j\|_{\lebe^\infty}\leq c\liminf_{j\rightarrow\infty}|\A u_j|(\R^n)=c|\A u|(\R^n).
\end{align*}
The proof of the equivalence is complete.
\end{proof}
We compare our result with \cite[Thm.~1.3]{PVS}, which states that if $n\geq 2$ and $u\in\sobo^{1,1}(\R^n)$ and $\partial_1\ldots\partial_nu\in\mathcal{M}(\R^n)$, then $u$ has a continuous representative. This strenghtens an earlier result, attributable to \textsc{Tartar} \cite[Thm.~1.2]{PVS}, which states that if $n\geq 2$, $D^nu\in\mathcal{M}$ and $u\in\sobo^{n-1,1}(\R^n)$, then $u$ has a continuous representative. One can speculate that, in our case, if $n\geq 2$, $u\in\sobo^{n-1,1}(\R^n)$ and $\A u\in\mathcal{M}$ for elliptic $\A$ of order $n$ that satisfies \eqref{eq:mistery_cond2}, then the embedding of $\bv^\A$ in $\lebe^\infty$ can be improved to continuity. This is not true, as we illustrate with an example below. In particular, it must be that the results in \cite{PVS} rely on a particular feature of the operator $\partial_1\ldots\partial_n$ on $\R^n$ (possibly, the availability of the Fundamental Theorem of Calculus in that case), rather than on a general self--improvement of the embedding of $\bv^\A$ in $\lebe^\infty$.
\begin{example}[{\cite[pp.18-19]{R_survey}}]
Let $\A=\Delta\circ(\di,\curl)$ on $\R^3$. Then $\A$ is elliptic (hence satisfies \eqref{eq:mistery_cond2}), non--canceling, and
\begin{align*}
\A u=\delta_0 e_1,\qquad\text{where}\qquad u(x)\coloneqq\dfrac{x}{|x|}.
\end{align*}
In particular, $u\in\bv^\A_{\locc}(\R^3)$ (so that $u\in\sobo^{2,1}_{\locc}(\R^3,\R^3)$).
\end{example}
\begin{proof}
All claims follow by direct computation, so that we can afford to present a streamlined proof. Writing $\B\coloneqq(\di,\curl)$, we note that $\B^*\circ\A=\Delta^2$, which has fundamental solution (proportional to) $|\cdot|$ in dimension $3$. Ellipticity of $\A$ follows by basic set theory as a composition of elliptic operators. Weak cancellation follows since all elliptic operators in odd dimensions are weakly canceling. Non--cancellation follows (with intersection equal to $\R e_1$) by direct computation; alternatively, one can compute that $\B(x/|x|^3)=\delta_0e_1$. We next write
\begin{align*}
u(x)=\B^*\circ\Delta^{-2}(\delta_0 e_1)=\B^*(|x|e_1)=D |x|=\dfrac{x}{|x|}.
\end{align*}
The Sobolev regularity of $u$ follows from Lemma~\ref{lem:ell=>sob}.
\end{proof}
In our setup ($\A$ elliptic of order $n$ on $\R^n$), the ideas of this paper are easily used to show that, in general, $\bv^\A_{\locc}(\R^n)\subset\hold(\R^n,V)$ implies that $\A$ is canceling. We do not know whether the converse is true, but speculate that this is the case.
\subsection{Characterization of condition~\eqref{eq:mistery_cond2}}\label{sec:char_mist_cond}
When proving that condition~\eqref{eq:mistery_cond2} is necessary for the embedding \eqref{eq:VS_k=n}, we in fact showed that, if \eqref{eq:mistery_cond2} fails for $w$, there exists $u\in\lebe^1_{\locc}(\R^n,V)$ such that $\A  u=\delta_0w$, but $D^{k-n}u$ is unbounded near zero. This property is actually equivalent with the failure of condition~\eqref{eq:mistery_cond2} for elliptic operators:
\begin{proposition}\label{prop:char_mist_cond}
Let $\A$ as in \eqref{eq:Ak} be elliptic, $k\geq n$. Then $\A$ satisfies condition~\eqref{eq:mistery_cond2} if and only if for $u\in\mathscr{S}'(\R^n,V)$, $w\in W\setminus\{0\}$ such that $\A  u=\delta_0w$, we have that $D^{k-n}u\in\lebe^\infty_{\locc}$.
\end{proposition}
\begin{proof}
We need only prove one implication. Suppose that condition~\eqref{eq:mistery_cond2} holds and let $0\neq w\in W$ (if any exist) be such that $\A  u=\delta_0w$. We consider the map 
\begin{align*}
f(\xi)\coloneqq\A^\dagger[\xi]w\otimes^{k-n}\xi,
\end{align*}
so $f$ is smooth (rational) away from zero and $(-n)$--homogeneous, hence a distribution in $\R^n\setminus\{0\}$. By \cite[Thm.~3.2.4]{HGOD1} and condition~\eqref{eq:mistery_cond2}, we have that $f$ defines a $(-n)$--homogeneous distribution in $\R^n$. By \cite[Thm.~7.1.18]{HGOD1}, we have that $f\in\mathscr{S}^\prime$. By \cite[Thm.~7.1.16]{HGOD1}, we have that $\check{f}$ is $0$--homogeneous and smooth away from zero. It is easy to see that then $\check{f}=D^{k-n}v$, where $v\in\mathscr{S}^\prime$ is smooth away from zero and
\begin{align*}
\widehat{v}(\xi)=\A^\dagger[\xi]w,
\end{align*}
so that $\A  v=\delta_0w$. So $u-v$ is $\A $--free, hence, by ellipticity, $u$ differs from $v$ by a smooth map (analytic, even). Since $D^{k-n}v=\check{f}$ is bounded, it follows that $D^{k-n}u$ is locally bounded.
\end{proof}
It seems relevant to compare the result of Proposition~\ref{prop:char_mist_cond} to the corresponding result for \textsc{Van Schaftingen}'s embeddings \eqref{eq:VS_j}, which was essentially already covered in Section~\ref{sec:EC}. We give another, very streamlined, proof.
\begin{proposition}[{\cite[Prop.~5.5]{VS}}]\label{prop:VS_canc}
Let $\A $ as in \eqref{eq:Ak} be elliptic, $1\leq j\leq\min\{k,n-1\}$. If there exists $u\in\mathscr{S}'(\R^n,V)$, $w\in W\setminus\{0\}$ such that $\A  u=\delta_0w$, then $D^{k-j}u\notin\lebe^{n/(n-j)}_{\locc}$.
\end{proposition}
\begin{proof}
As above, we define $f(\xi)\coloneqq\A^\dagger[\xi]w\otimes^{k-j}\xi$, which is $(-j)$--homogeneous, so we can apply \cite[Thm.~3.2.3, 7.1.16, 7.1.18]{HGOD1}\footnote{It is important to mention that \cite[Thm.~3.2.4]{HGOD1}, which we used to prove Proposition~\ref{prop:char_mist_cond}, is essentially the degenerate version of \cite[Thm.~3.2.3]{HGOD1}, which we use here. In a nutshell, this is the basic difference between the canceling condition and condition~\eqref{eq:mistery_cond2}.} to show that $\check{f}\in\mathscr{S}'$ is $(j-n)$--homogeneous, so either $\check{f}\equiv0$, or $\check{f}\notin\lebe^{n/(n-j)}_{\locc}$. As above, $\check{f}=D^{k-j}v$, where $u-v$ is $\A $--free, hence a smooth map. Also $f\equiv0$ implies $w=0$, so the conclusion follows.
\end{proof}
Although for $0\leq j\leq n\leq k$ the inequalities with $u\in\hold^\infty_c(\R^n,V)$
\begin{align}\label{eq:j=0...n}
\|D^{k-j}u\|_{\lebe^{n/(n-j)}}\lesssim\|\A u\|_{\lebe^1}
\end{align}
rely on different conditions for $1\leq j\leq n-1$ and $j=n$, Propositions~\ref{prop:char_mist_cond} and \ref{prop:VS_canc} reveal a phenomenological similarity: if the ``worst'' measures from the point of view of convolution with $D^{k-j}\left(\mathscr{F}^{-1}\A^\dagger[\cdot]\right)$ 
do not lie in 
\begin{align*}
\{\A  u\text{ is a bounded measure}\},
\end{align*}
then the embeddings hold. This is in sharp contrast with the case $j=0$, when \eqref{eq:j=0...n} holds only for trivial $\A $ by Ornstein's Non--inequality.

Also, if $\A$ is elliptic, satisfies condition~\eqref{eq:mistery_cond2}, but is not canceling (see Section~\ref{sec:examples_appdiff} for examples), then the embedding \eqref{eq:j=0...n} holds \emph{only} for $j=n$. There exists a single $\bv^\A_{\locc}$--map $u$ such that $D^{k-j}u\notin\lebe^{n/(n-j)}_{\locc}$ for $j=0,\ldots, n-1$, but $D^{k-n}u\in\lebe^\infty$.

Before moving on to examples, we use Proposition~\ref{prop:char_mist_cond} to write out Theorem~\ref{thm:main_infty} in an invariant form, independent of the assumption made in the Introduction that $\R^n$, $V$, $W$ have Euclidean structure:
\begin{remark}\label{rk:invariance}
Assume that $\R^n$, $V$, $W$ are (viewed as) normed finite dimensional vector spaces, i.e., no inner product is defined a priori. We claim that, for an elliptic operator $\A $ of order $k\geq n$, we have that
\begin{align*}
\|D^{k-n}u\|_{\lebe^\infty(\R^n,V\odot^{k-n}\R^n)}\lesssim \|\A u\|_{\lebe^1(\R^n,W)}
\end{align*}
if and only if
\begin{align}\label{eq:inv_cond}
\A u=\delta_0w\implies D^{k-n}u\in\lebe^\infty_{\locc}(\R^n,V\odot^{k-1}\R^n).
\end{align}
for $u\in\mathscr{S}^\prime(\R^n,V)$, $w\in W\setminus\{0\}$.
\end{remark}
\begin{proof}
We choose bases of  $\R^n$, $V$, $W$. Defining inner products respectively such that these bases are orthonormal, we obtain new, equivalent norms and an equivalent Lebesgue measure on $\R^n$.

By the argument used to prove Proposition~\ref{prop:char_mist_cond}, we have that \eqref{eq:inv_cond} is represented in the new coordinates as condition~\eqref{eq:mistery_cond2}. By the proof of Theorem~\ref{thm:main_infty} in Section~\ref{sec:proof_infty}, we have that the inequality holds, with respect to the new norms and Lebesgue measure. By finite dimensionality of $\R^n$, $V$, $W$, the inequality follows for the initial norms and Lebesgue measure.
\end{proof}

\subsection{Examples}\label{sec:examples_appdiff}
In this section, we will show by way of example that condition~\eqref{eq:mistery_cond2} is strictly weaker than the canceling condition for elliptic operators. Throughout, with a slight abuse of notation, $\{e_j\}_j$ will denote the standard Euclidean basis in any finite dimensional space we may consider.

As noted in the Introduction, condition~\eqref{eq:mistery_cond2} is automatically satisfied in odd dimensions for elliptic operators. We have:
\begin{example} Let $n=2d+1$. Then the operator $\A \coloneqq\Delta^d\circ(\di,\curl)$ on $\R^n$ from $\R^n$ to $\R\times\R^{n\times n}_{\asym}$ is elliptic and non--canceling.
\end{example}
\begin{proof}
We write $\mathbb{B}\coloneqq(\di,\curl)$. An elementary computation shows that, for non--zero $\xi$, $\curl[\xi]v=0$ implies that $v=\alpha\xi$ for some $\alpha\in\R$. If also $0=\di[\xi]v=\alpha|\xi|^2$, then $v=0$, so $\mathbb{B}$ is elliptic. By set theory, so is $\A $. To see that $\A $ is non--canceling, note from the previous calculation that $\A[\xi]\xi=|\xi|^2e_1$.
\end{proof}

It may be that \eqref{eq:VS_k=n} holds for all elliptic operators. This is not the case by:
\begin{example} Let $n=2d$. Then the operator $\A \coloneqq\Delta^d$ on $\R^n$ from $\R$ to $\R$ is elliptic and fails condition~\eqref{eq:mistery_cond2}.
\end{example}
\begin{proof}
Ellipticity is obvious. Failure of condition~\eqref{eq:mistery_cond2} is equivalent to
\begin{align*}
\int_{\mathbb{S}^{n-1}}\dfrac{1}{|\xi|^{n}}\dif\mathscr{H}^{n-1}(\xi)\neq0,
\end{align*}
which is clearly true.
\end{proof}

On the other hand, it may be that~\eqref{eq:mistery_cond2} is not weaker than the canceling condition in even dimensions. This is also not the case by:
\begin{example}\label{ex:simple}
The operator
\begin{align}\label{eq:example}
\A_1u\coloneqq
\left(
\begin{matrix}
(\partial_1^2-\partial^2_2)u_1+2\partial_1\partial_2u_2\\
-2\partial_1\partial_2u_1+(\partial_1^2-\partial^2_2)u_2
\end{matrix}
\right),
\end{align}
on $\R^2$ from $\R^2$ to $\R^2$ is elliptic, non--canceling, and satisfies condition~\eqref{eq:mistery_cond2}.
\end{example}
\begin{proof}
Ellipticity is in this case, equivalent to invertibility for $\xi\neq0$ of the matrix
\begin{align*}
\A_1[\xi]=
\left(
\begin{matrix}
\xi_1^2-\xi^2_2&2\xi_1\xi_2\\
-2\xi_1\xi_2&\xi_1^2-\xi^2_2
\end{matrix}
\right),
\end{align*}
which has determinant equal to $|\xi|^4$. Non--cancellation follows from ellipticity. To check condition~\eqref{eq:mistery_cond2}, we have that
\begin{align*}
\A_1^\dagger[\xi]=\A_1^{-1}[\xi]=\dfrac{1}{|\xi|^4}
\left(
\begin{matrix}
\xi_1^2-\xi^2_2&-2\xi_1\xi_2\\
2\xi_1\xi_2&\xi_1^2-\xi^2_2
\end{matrix}
\right),
\end{align*}
so that one computes
\begin{align*}
\int_{\mathbb{S}^1}\xi_1^2-\xi_2^2\dif\mathscr{H}^1(\xi)=\int_0^{2\pi}\cos(2\theta)\dif\theta=0,\,\int_{\mathbb{S}^1}2\xi_1\xi_2\dif\mathscr{H}^1(\xi)=\int_0^{2\pi}\sin(2\theta)\dif\theta=0
\end{align*}
to conclude the proof.
\end{proof}
So far, above or in \cite{BVS}, we have only displayed examples of $\A $ for which the inequality~\eqref{eq:VS_k=n} holds which are either canceling or for which $\mathcal{L}\equiv0$. The following example shows that condition~\eqref{eq:mistery_cond2} cannot be simplified easily, as the interaction between $\mathcal{L}$ and the intersection of images can occur in a non--trivial way. We augment the previous example as follows:
\begin{example}
The operator
\begin{align*}
\A_2
\left(
\begin{matrix}
u_1\\u_2\\u_3
\end{matrix}
\right)
\coloneqq
\left(
\begin{matrix}
\A_1
\left(
\begin{matrix}
u_1\\
u_2
\end{matrix}
\right)
\\
\partial^2_1u_3\\
\sqrt{2}\partial_1\partial_2u_3\\
\partial^2_2u_3
\end{matrix}
\right)
\end{align*}
on $\R^2$ from $\R^3$ to $\R^5$ is elliptic, non--canceling, and satisfies condition~\eqref{eq:mistery_cond2} with $\mathcal{L}\not\equiv0$. Here $\A_1$ is as in \eqref{eq:example}.
\end{example}
\begin{proof}
Consider, more generally, an operator
\begin{align*}
\A[\xi]\coloneqq\left(
\begin{matrix}
\mathbb{B}_1[\xi]&0\\
0&\mathbb{B}_2[\xi]
\end{matrix}
\right),
\end{align*}
where $\mathbb{B}_1$ is elliptic and square on $\R^n$ from $V_1$ to $V_1$ and $\mathbb{B}_2$ is elliptic and canceling on $\R^n$ from $V_2$ to $W_2$. By matrix multiplication
\begin{align}\label{eq:square_matrix}
\A^*[\xi]\A[\xi]=\left(
\begin{matrix}
\mathbb{B}_1^*[\xi]\mathbb{B}_1[\xi]&0\\
0&\mathbb{B}^*_2[\xi]\mathbb{B}_2[\xi]
\end{matrix}
\right),
\end{align}
so that $\det(\A^*[\xi]\A[\xi])=\det(\mathbb{B}_1^*[\xi]\mathbb{B}_1[\xi])\det(\mathbb{B}_2^*[\xi]\mathbb{B}_2[\xi])\neq0$ for $\xi\neq0$. In particular, $\A $ is elliptic. We also have that
\begin{align*}
\A[\xi]
\left(
\begin{matrix}
\mathbb{B}_1^{-1}[\xi]v_1\\
0
\end{matrix}
\right)
=
\left(
\begin{matrix}
v_1\\
0
\end{matrix}
\right)
\end{align*}
for $0\neq\xi\in\R^n$, $0\neq v_1\in V_1$, so $\A $ is non--canceling.

For later purposes we also compute $\bigcap_{\xi\in\mathbb{S}^{n-1}}\mathrm{im\,}\A[\xi]$, which, by the above computation contains $V_1\times\{0\}$. It is easy to see that equality holds: let
\begin{align*}
\left(
\begin{matrix}
v_1\\
v_2
\end{matrix}
\right)
\in\bigcap_{\xi\in\mathbb{S}^{n-1}}\mathrm{im\,}\A[\xi],
\end{align*}
so that, again by matrix multiplication, $v_2\in\bigcap_{\xi\in\mathbb{S}^{n-1}}\mathrm{im\,}\mathbb{B}_2[\xi]=\{0\}$.

We now choose $\mathbb{B}_1\coloneqq\A_1$ and
\begin{align*}
\mathbb{B}_2\coloneqq
\left(
\begin{matrix}
\partial^2_1\\
\sqrt{2}\partial_1\partial_2\\
\partial^2_2
\end{matrix}
\right)
\end{align*}
on $\R^2$ from $\R$ to $\R^3$. The operator $\A_2$ thus defined is elliptic and non--canceling. 

We next check that condition~\eqref{eq:mistery_cond2} holds, to which end we record that 
\begin{align*}
J\coloneqq\bigcap_{\xi\in\mathbb{S}^{n-1}}\mathrm{im\,}\A[\xi]=\R^2\times\{0\}
\end{align*}
from the above considerations. By \eqref{eq:square_matrix}, we have that $\A^*[\xi]\A[\xi]=|\xi|^4\id$, so that $\A^\dagger[\xi]=|\xi|^{-4}\A^*[\xi]$. To check that $\mathcal{L}(J)=\{0\}$ is thus equivalent to
\begin{align*}
\int_{\mathbb{S}^1}
\left(
\begin{matrix}
\xi_1^2-\xi^2_2&-2\xi_1\xi_2\\
2\xi_1\xi_2&\xi_1^2-\xi^2_2
\end{matrix}
\right)
\dif\mathscr{H}^1(\xi)=0,
\end{align*}
which we know to be true from Example~\ref{ex:simple}.

To check that $\R^{3\times5}\ni\mathcal{L}\not\equiv0$, we look at
\begin{align*}
\mathcal{L}_{33}=\int_{\mathbb{S}^1}\dfrac{\xi_1^2}{|\xi|^4}\dif\mathscr{H}^1(\xi)>0.
\end{align*}
The proof is complete.
\end{proof}
\section{Proof of Theorem~\ref{thm:main_k=1}}\label{sec:k=1}

For clarity of exposition, we present the proofs for the order $k=1$ case separately. We begin by showing sufficiency of EC for the Lorentz--differentiability. The outline of the following Lemma is that one can use the Sobolev--type embedding \cite[Thm.~8.5]{VS} (i.e., \eqref{eq:VS_appdiff} for Lorentz spaces) to boost the main result \cite[Thm.~3.4]{ABC}, which gives $\lebe^1$--differentiability $\mathscr{L}^n$--a.e. of maps in $\bv^\A$ for elliptic $\A$.
\begin{lemma}\label{lem:suff_EC}
Let $n>1$, $\A$ as in \eqref{eq:A} be EC, and $1<q<\infty$. Then all maps in $\bv^\A_{\locc}$ are $\lebe^{n/(n-1),q}$--differentiable $\mathscr{L}^n$--a.e.
\end{lemma}
\begin{proof}
We begin by deriving the estimate 
\begin{align}\label{eq:key_est}
\|v\|_{\lebe^{n/(n-1),q}(\ball_r(x))}\lesssim|\A v|\left(\overline{\ball_{2r}(x)}\right)+r^{-1}\|v\|_{\lebe^1(\ball_{2r}(x))}
\end{align}
for all $v\in\bv^\A_{\locc}$. Consider a sequence $v_j$ of mollifications of $v$ (which converge $\A$--strictly to $v$ in $\bv^\A_{\locc}$ \cite[Thm.~2.8]{BDG}) and a positive cut-off function $\rho\in\hold^\infty_c(\ball_{2r}(x))$ such that $\rho=1$ in $\ball_r(x)$ and $|\nabla\rho|\lesssim r^{-1}$. By Fatou's Lemma for Lorentz spaces and \cite[Thm.~8.5]{VS} we have:
\begin{align*}
\text{LHS}&\leq\liminf_j\|\rho v_j\|_{\lebe^{n/(n-1),q}(\ball_{2r}(x))}\\
&\lesssim\liminf_j\|\A(\rho v_j)\|_{\lebe^1(\ball_{2r}(x))}\\
&\lesssim\liminf_j\left(\|\rho\A v_j\|_{\lebe^1(\ball_{2r}(x))}+r^{-1}\| v_j\|_{\lebe^1(\ball_{2r}(x))}\right)\\
&\leq\liminf_j\left(\|\A v_j\|_{\lebe^1(\ball_{2r}(x))}+r^{-1}\| v_j\|_{\lebe^1(\ball_{2r}(x))}\right)\\
&=\text{RHS}.
\end{align*}
We then fix $u\in\bv^\A_{\locc}$ and apply \eqref{eq:key_est} to $v\coloneqq R^1_x u\in\bv^\A_{\locc}$. Recall that $R^1_x u(y)=u(y)-u(x)-\nabla u(x)(y-x)$, where $x$ is chosen as a Lebesgue point of $\A u$ and a point of $\lebe^1$--differentiability of $u$, with approximate gradient $\nabla u(x)$ (such points exist $\mathscr{L}^n$--a.e. by Lemma~\ref{lem:algebra_Au} for $k=1$). We obtain
\begin{align*}
\|R^1_xu\|_{\lebe^{n/(n-1),q}(\ball_r(x))}\lesssim|\A R^1_x u|\left(\overline{\ball_{2r}(x)}\right)+r^{-1}\| R^1_x u\|_{\lebe^1(\ball_{2r}(x))}.
\end{align*}
The second term equals $o(r^n)$ as $r\downarrow0$ by $\lebe^1$--differentiability of $u$ at $x$. The first term equals $|\A u-A(\nabla u(x))|(\ball_{2r}(x))=o(r^n)$ as $r\downarrow0$ by Lebesgue differentiation for Radon measures and \eqref{eq:algebra_Au}. The proof is complete.
\end{proof}

\begin{lemma}\label{lem:nec_ell_appdiff}
Let $\A$ be as in \eqref{eq:A}, $1<p<\infty$, $1\leq q\leq\infty$. Suppose that all maps in $\bv^\A_{\locc}$ are $\lebe^{p,q}$--differentiable $\mathscr{L}^n$--a.e. Then $\A$ is elliptic.
\end{lemma}
\begin{proof}
By making $p$ smaller if necessary, we can assume that maps in $\bv^\A_{\locc}$ are in fact $\lebe^p$--differentiable, with $p>1$.

We assume that $\A$ is not elliptic, so there exist non--zero $\xi\in\R^n$, $v\in V$ such that $\A[\xi]v=0$. For $f\in\lebe^1_{\locc}(\R)$ and $u(x)\coloneqq f(x\cdot\xi)v$ for $x\in\R^n$, a simple computation using \eqref{eq:A} shows that $\A u=0$ in $\mathscr{D}^\prime$. 
We let $g\in\lebe^1(\R)$ be a \emph{positive} map that is not $\lebe^p$--integrable in any neighbourhood of $0$, e.g., $g(t)=\rho(t)|t|^{-1/p}$ for some $\rho\in\hold^\infty_c(\R,[0,1])$ that equals one in a neighbourhood of zero. We let $\{x_l\}_{l=1}^\infty$ be dense in $\R^n$ and define
\begin{align}\label{eq:plane_wave}
u(x)\coloneqq \sum_{l=1}^\infty2^{-l}g((x-x_l)\cdot\xi)v,
\end{align}
so that we have $\A u=0$. There is no loss of generality in assuming that $|v|=1$ and $\xi=e_n$. For a cube $Q$ with sides parallel to the coordinate axes, we have that
\begin{align*}
\|u\|_{\lebe^1(Q)}&=\sum_{l=1}^\infty 2^{-l}\int_{Q}g((x-x_l)\cdot e_n)\dif x\leq \ell(Q)^{n-1}\sum_{l=1}^\infty 2^{-l}\int_{\R}g(t)\dif t<\infty,
\end{align*}
where $\ell(Q)$ denotes the side length of $Q$. In particular, $u\in\bv^\A_{\locc}$. On the other hand, in a neighbourhood of $Q_r(x_k)$ we have, by positivity of $g$,
\begin{align*}
\|u\|^p_{\lebe^p(Q_r(x_k))}&=\int_{Q_r(x_k)}\left|\sum_{j=1}^\infty 2^{-j}g((x-x_j)\cdot e_n)\right|^p\dif x\\
&\geq2^{-kp}\int_{Q_r(x_k)}g((x-x_k)\cdot e_n)^p\dif x\\
&=2^{-kp}(2r)^{n-1}\int_{-r}^r g(t)^p\dif t,
\end{align*}
so $u$ is not $p$--integrable in any open set, hence cannot be $\lebe^p$--differentiable. This contradiction completes the proof.
\end{proof}
To complete the proof of Theorem~\ref{thm:main_k=1}, it remains to see that for elliptic first order operators $\A$, cancellation is necessary for $\lebe^{n/(n-1),q}$--differentiability. This will require the closer look at the canceling condition that we took in Lemma~\ref{lem:canc_char}.
\begin{lemma}\label{lem:nec_canc}
Let $\A$ as in \eqref{eq:A} be elliptic, $1\leq q<\infty$ and suppose that maps in $\bv^\A_{\locc}$ are $\lebe^{n/(n-1),q}$--differentiable $\mathscr{L}^n$--a.e. Then $\A$ is canceling.
\end{lemma}
\begin{proof}
We assume that $\A$ is not canceling, and let $u_h\in\bv^\A_{\locc}$ be as in the proof of Lemma~\ref{lem:canc_char}, so $u_h\notin\lebe^{n/(n-1),q}(\ball_r(0))$ for any $r>0$. 

We will formulate a Baire category argument to prove existence of a $\bv^\A$--map that is not in $\lebe^{n/(n-1),q}$ in any open set, which clearly suffices to prove the claim. For a ball $\ball\subset\R^n$, let $\mathfrak{X}\coloneqq \bv^\A(\ball)$, which is a complete metric space under the metric induced by the obvious norm, let $\{x_l\}_{l=1}^\infty$ be a dense subset and consider the subsets
\begin{align*}
X_{l,m,L}\coloneqq \{u\in \mathfrak{X}\colon \|u\|_{\lebe^{n/(n-1),q}(\ball(x_l,m^{-1})\cap\ball)}\leq L\}
\end{align*}
for positive integers $l,m,L$. We also make the convention 
\begin{align*}
X_{l,m,\infty}\coloneqq \{u\in \mathfrak{X}\colon \|u\|_{\lebe^{n/(n-1),q}(\ball(x_l,m^{-1})\cap\ball)}<\infty\},
\end{align*}
which is a vector subspace of $\mathfrak{X}$.

It is then easy to check that the sets $X_{l,m,L}$ are nowhere dense. Let $u_i\in X_{l,m,L}$ be such that $u_i\rightarrow u$ in $\mathfrak{X}$. Then $u_i$ converges pointwisely to $u$ on a subsequence which we do not relabel. By the Fatou property in Lorentz spaces we have that
\begin{align*}
\|u\|_{\lebe^{n/(n-1),q}(\ball(x_l,m^{-1})\cap\ball)}\leq\liminf_{i\rightarrow\infty}\|u_i\|_{\lebe^{n/(n-1),q}(\ball(x_l,m^{-1})\cap\ball)}\leq L,
\end{align*}
so $X_{l,m,L}$ is closed. Since $X_{l,m,L}\subset X_{l,m,\infty}$, which is a proper subspace of $\mathfrak{X}$, as $u_h(\cdot-x_l)\notin\lebe^{n/(n-1),q}(\ball(x_l,m^{-1})\cap\ball)$, we have that $X_{l,m,L}$ have empty interior.

By the Baire's Category Theorem, we obtain existence of a map in $\bv^\A(\ball)$, that is not in any $X_{l,m,L}$, hence in no $X_{l,m,\infty}$. The proof is complete.
\end{proof}

\section{Proof of Theorem~\ref{thm:main_k_diff}}\label{sec:k}

We extend the ideas from Section~\ref{sec:k=1}. We begin with sufficiency of the algebraic conditions for the claims of Theorem~\ref{thm:main_k_diff}. We make the convention that $n/0=\infty$.
\begin{lemma}
Let $\A$ as in \eqref{eq:Ak} be elliptic, $1\leq j\leq \min\{n,k\}$. Suppose that
\begin{enumerate}
\item If $1\leq j\leq \min\{n-1,k\}$ we have that $\A$ is canceling;
\item If $j=n\leq k$ we have that $\A$ satisfies condition~\eqref{eq:mistery_cond2}.
\end{enumerate}
Then for all $u\in\bv^\A_{\locc}$ we have that $\D^{k-j}u\in\taylor^{j,n/(n-j)}(x)$ for $\mathscr{L}^n$--a.e. $x\in\R^n$.
\end{lemma}
\begin{proof}
We choose $x\in\R^n$ such that, by Lemmas~\ref{lem:ell=>sob}, \ref{lem:CZ_lemma}, \ref{lem:algebra_Au}, we have that $\D^{k-j}u\in\taylor^{j,1}(x)$ for $j=1\ldots k$, where $\D^{k-j}u\in\lebe^1_{\locc}$ are weak derivatives of $u$, and \eqref{eq:algebra_Au} holds. In particular, we have a $k$--th order $\lebe^1$--Taylor Polynomial $P^k_x u$ of $u$ at $x$. Uniqueness of $\lebe^p$--Taylor polynomials implies that $\D^{k-j}P^k_x u=P^{j}_x \D^{k-j}u$ and that $\A P^k_xu=A(\nabla^k u(x))$. The set of such $x$ has $\mathscr{L}^n$--null complement.

We fix $1\leq j\leq \min\{n-1,k\}$, so that $\A$ is canceling, and estimate: 
\begin{align}\label{eq:key_est_k}
\|\D^{k-j}v\|_{\lebe^{n/(n-j)}(\ball_r(x))}\lesssim|\A v|(\overline{\ball_{2r}(x)})+\sum_{l=1}^k r^{-l}\|\D^{k-l}v\|_{\lebe^1(\ball_{2r}(x))},
\end{align}
which holds for all $v\in\bv^\A_{\locc}$. This is obtained analogously to \eqref{eq:key_est}, using \eqref{eq:VS_j}  instead of \cite[Thm.~8.5]{VS}. We apply \eqref{eq:key_est_k} to $R^k_xu\coloneqq u-P^k_xu$ to get
\begin{align*}
\left(\int_{\ball_r(x)}|\D^{k-j}R^k_xu|^{n/(n-j)}\dif y\right)^{(n-j)/n}&\lesssim|\A u-\A P^k_x u|(\overline{\ball_{2r}(x)})\\
&+\sum_{l=1}^{k}\frac{1}{r^l}\int_{\ball_{2r}(x)}|\D^{k-l}u-\D^{k-l}P^k_xu|\dif y\\
&=|\A u-A(\nabla^ku(x))|(\overline{\ball_{2r}(x)})\\
&+\sum_{l=1}^{k}\frac{1}{r^l}\int_{\ball_{2r}(x)}|\D^{k-l}u-P^l_x\D^{k-l}u|\dif y.
\end{align*}
To obtain the correct scaling, we multiply by $r^{j-n}$ to get
\begin{align*}
\left(\fint_{\ball_r(x)}|R^j_x\D^{k-j}u|^{n/(n-j)}\dif y\right)^{(n-j)/n}&\lesssim r^j\dfrac{|\A u-\A^{ac}u(x)|(\overline{\ball_{2r}(x)})}{r^n}\\
&+\sum_{l=1}^{k}r^{j-l}\fint_{\ball_{2r}(x)}|R^l_x\D^{k-l}u|\dif y.
\end{align*}
As $r\downarrow0$, the first term equals $o(r^j)$ by Lebesgue differentiation, whereas the averaged integrals equal $o(r^l)$ for all $l=1\ldots k$. The first case is proved.

If $j=n\leq k$, we suppose that $\A$ satisfies condition~\eqref{eq:mistery_cond2} and replace \eqref{eq:key_est_k} with
\begin{align}\label{eq:key_est_k_large}
\|\D^{k-n}v\|_{\lebe^\infty(\ball_r(x))}\lesssim|\A v|(\overline{\ball_{2r}(x)})+\sum_{l=1}^k r^{-l}\|\D^{k-l}v\|_{\lebe^1(\ball_{2r}(x))}
\end{align}
for $v\in\bv^\A_{\locc}$. This, again, follows just as \eqref{eq:key_est}, by using Theorem~\ref{thm:main_infty} instead of \eqref{eq:VS_j}. We proceed as above, with $v\coloneqq R^k_x u$, so that we have
\begin{align*}
\|R^n_x\D^{k-n}u\|_{\lebe^\infty(\ball_r(x))}\lesssim r^n\frac{|\A u-\A^{ac}u(x)|(\overline{\ball_{2r}(x)})}{r^n}+\sum_{l=1}^k r^{n-l}\fint_{\ball_{2r}(x)}|R^l_x\D^{k-l}u|
,
\end{align*}
and we can conclude as in the previous case.
\end{proof}
We next prove necessity of ellipticity for both Theorems~\ref{thm:main_k_diff}, \ref{thm:sub_crit}. This is contained in the following Lemma, which is proved by an inexpensive modification of the proof of Lemma~\ref{lem:nec_ell_appdiff}.
\begin{lemma}\label{lem:nec_ell_k}
Let $\A$ be as in \eqref{eq:Ak}. Suppose that there exist $1\leq j\leq k$, $1<p<\infty$ such that all $u\in\bv^\A_{\locc}$ are such that $\D^{k-j}u\in\taylor^{j,p}(x)$ for $\mathscr{L}^n$--a.e. $x\in\R^n$. Then $\A$ is elliptic.
\end{lemma}
\begin{proof}
Suppose that $\A$ is not elliptic, and let $\xi$, $v$ be as in the proof of Lemma~\ref{lem:nec_ell_appdiff}. We let $g(t)=|t|^{k-j-1/p}(\mathrm{sgn}(t))^{k-j}$ for $t\in\R$, so $g\in\lebe^1_{\locc}(\R)$, and define $u$ by \eqref{eq:plane_wave}. It is clear that $\A u=0$ and that $u\in\lebe^1_{\locc}$ (as in the proof of Lemma~\ref{lem:nec_ell_appdiff}), so $u\in\bv^\A_{\locc}$. A simple computation shows that 
\begin{align*}
\D^{k-j}u(x)&=\sum_{l=1}^\infty 2^{-l} \dfrac{\dif^{k-j} g}{\dif t^{k-j}}((x-x_l)\cdot\xi)v\otimes^{k-j}\xi\\
&=c(p,j,k)\sum_{l=1}^\infty 2^{-l}|(x-x_l)\cdot\xi|^{-1/p}v\otimes^{k-j}\xi,
\end{align*}
which, by an argument as in the proof of Lemma~\ref{lem:nec_ell_appdiff}, is nowhere $\lebe^p$--integrable. The proof is complete.
\end{proof}
We also generalize Lemma~\ref{lem:nec_canc}, to obtain necessity of cancellation for Theorem~\ref{thm:main_k_diff}\ref{itm:a}. Again, Lemma~\ref{lem:canc_char} will prove to be instrumental.
\begin{lemma}\label{lem:nec_canc_k}
Let $\A$ as in \eqref{eq:Ak} be elliptic. Suppose that for  some $1\leq j\leq\min\{k,n-1\}$ we have that all $u\in\bv^\A_{\locc}$ are such that $\D^{k-j}u\in\taylor^{j,n/(n-j)}(x)$ for $\mathscr{L}^n$--a.e. $x\in\R^n$. Then $\A$ is canceling.
\end{lemma}
\begin{proof}
We assume that $\A$ is not canceling and consider $u_h\in\bv^\A_{\locc}$ as in the proof of Lemma~\ref{lem:canc_char}, so that $\D^{k-j}u_h\notin\lebe^{n/(n-j)}(\ball_r(0))$ for any $r>0$. As before, we work in $\mathfrak{X}\coloneqq \bv^\A(\R^n)$, which is a complete metric space in the metric induced by the norm $\|\cdot\|_{\bv^\A(\R^n)}\coloneqq \|\cdot\|_{\lebe^1(\R^n)}+|\A\cdot|(\R^n)$. Let $\{x_l\}_{l=1}^\infty$ be a dense subset of $\R^n$ and consider the subsets
\begin{align*}
Y_{l,m,L}\coloneqq \{u\in \mathfrak{X}\colon \|\D^{k-j}u\|_{\lebe^{n/(n-j)}(\ball(x_l,m^{-1}))}\leq L\}
\end{align*}
for positive integers $l,m,L$, and the subspace
\begin{align*}
Y_{l,m,\infty}\coloneqq \{u\in \mathfrak{X}\colon \|\D^{k-j}u\|_{\lebe^{n/(n-j)}(\ball(x_l,m^{-1}))}<\infty\}.
\end{align*}
We check that the sets $Y_{l,m,L}$ are nowhere dense. Let $u_i\in Y_{l,m,L}$ be such that $u_i\rightarrow u$ in $\mathfrak{X}$. By standard mollification arguments, we have that $\hold_c^\infty(\R^n,V)$ is ($\A$--)strictly dense in $\mathfrak{X}$ \cite[Thm.~2.8(b)]{BDG}, so that maps in $\mathfrak{X}$ can be represented by \eqref{eq:conv_rep}. In particular, Lemma~\ref{lem:conv} implies that
\begin{align*}
\D^{k-j}v=\D^{k-j}K^\A\star\A v
\end{align*}
for $v\in \mathfrak{X}$, where $\D^{k-j}K^\A$ is $(j-n)$--homogeneous. Standard boundedness of Riesz potentials \cite[Thm.~V.1]{Stein} and inclusion of Lorentz spaces imply that, for any $1\leq p<n/(n-j)$ and any ball $B\subset\R^n$,
\begin{align}\label{eq:L1_loc_conv}
\|\D^{k-j}v\|_{\lebe^p(B)}\leq C(B)|\A v|(\R^n)
\end{align}
for any $v\in\mathfrak{X}$ (here, it is crucial that $0<j<n$). Thus $\D^{k-j}u_i$ converges $\mathscr{L}^n$--a.e. to $\D^{k-j}u$ on a subsequence which we do not relabel. By Fatou's Lemma we have:
\begin{align*}
\|\D^{k-j}u\|_{\lebe^{n/(n-j)}(\ball(x_l,m^{-1}))}\leq\liminf_{i\rightarrow\infty}\|\D^{k-j}u_i\|_{\lebe^{n/(n-j)}(\ball(x_l,m^{-1})}\leq L,
\end{align*}
so $Y_{l,m,L}$ is closed. Next, we consider cut--off functions $\rho_{l,m}\in\hold^\infty_c$ such that $\rho_{l,m}=1$ in $B(x_l,m^{-1})$. It follows that $\rho_{l,m}u_h(\cdot-x_l)\in \bv^\A(\R^n)$ by the regularity of $u_h$ away from zero and the Leibniz rule. However, $\D^{k-j}u_h(\cdot-x_l)\notin\lebe^{n/(n-j)}(\ball(x_l,m^{-1}))$, so that $Y_{l,m,\infty}$ is a proper subspace of $\mathfrak{X}$. Therefore, the sets $Y_{l,m,L}\subset Y_{l,m,\infty}$ have empty interior, so we can conclude by Baire's Category Theorem.
\end{proof}
It remains to establish necessity of condition~\eqref{eq:mistery_cond2} for Theorem~\ref{thm:main_k_diff}\ref{itm:b}. The proof is, of course, similar to that of Lemma~\ref{lem:nec_canc_k}.
\begin{lemma}\label{lem:nec_mist_cond}
Let $\A$ as in \eqref{eq:Ak} be elliptic, $k\geq n\geq1$. Suppose we have that all $u\in\bv^\A_{\locc}$ are such that $\D^{k-n}u\in\taylor^{n,\infty}(x)$ for $\mathscr{L}^n$--a.e. $x\in\R^n$. Then $\A$ satisfies condition~\eqref{eq:mistery_cond2}.
\end{lemma}
\begin{proof}
If $n=1$, condition~\eqref{eq:mistery_cond2} follows by ellipticity. Let $n>1$. We assume that $\A$ fails condition~\eqref{eq:mistery_cond2} and consider $u_h\in\bv^\A_{\locc}$ as in the proof of necessity for Theorem~\ref{thm:main_infty} (see also Proposition~\ref{prop:char_mist_cond}), so that $\D^{k-n}u_h\notin\lebe^{\infty}(\ball_r(0))$ for any $r>0$. We slightly modify the proof of Lemma~\ref{lem:nec_canc_k} by choosing, for $B\coloneqq B(0,1)$,
\begin{align*}
\mathfrak{X}\coloneqq\{u\in \bv^\A(\R^n)\colon \spt u\subset 2\bar B\},
\end{align*}
which is Banach if endowed with the subspace norm. Let $\{x_l\}_{l=1}^\infty$ be a dense subset of $B$ and consider the subsets
\begin{align*}
Z_{l,m,L}\coloneqq \{u\in \mathfrak{X}\colon \|\D^{k-n}u\|_{\lebe^{\infty}(\ball(x_l,m^{-1}))}\leq L\}
\end{align*}
for positive integers $m>1,\,l,L$, and the subspace
\begin{align*}
Z_{l,m,\infty}\coloneqq \{u\in \mathfrak{X}\colon \|\D^{k-n}u\|_{\lebe^{\infty}(\ball(x_l,m^{-1}))}<\infty\}.
\end{align*}
We check that the sets $Z_{l,m,L}$ are nowhere dense. Let $u_i\in Z_{l,m,L}$ be such that $u_i\rightarrow u$ in $\mathfrak{X}$.  
We collect from \eqref{eq:L1_loc_conv} with $j=n-1>0$ that, for any $1\leq p<n$ we have
\begin{align}\label{eq:Lp_loc_est}
\|\D^{k-n+1}v\|_{\lebe^p(B)}\leq C(B)|\A v|(\R^n)
\end{align}
for any  $v\in\mathfrak{X}$. By Lemma \ref{lem:conv} for $\A=D^{k-n+1}$ and $j=n-1$, we have that for a $(1-n)$--homogeneous, smooth convolution kernel $K$
\begin{align*}
|D^{k-n}v|=|K\star D^{k-n+1}v|\lesssim I_1|D^{k-n+1}v|
\end{align*}
for $v\in\mathfrak{X}$. By \cite[Lem.~7.12]{GT}, for any $1\leq q<np/(n-p)$ and \eqref{eq:Lp_loc_est} we have
\begin{align*}
\|D^{k-n}v\|_{\lebe^q(B)}\leq C(B)\|D^{k-n+1}v\|_{\lebe^p(\ball)}\leq C(B) |\A v|(\R^n)
\end{align*}
for $v\in\mathfrak{X}$. In particular, $\D^{k-n}u_i\rightarrow\D^{k-n} u$ in $\lebe^s_{\locc}$ for any $1\leq s<\infty$, so $\D^{k-n}u_i$ converges $\mathscr{L}^n$--a.e. to $\D^{k-n}u$ on a subsequence which we do not relabel. It follows that $\mathscr{L}^n$--a.e. in $B(x_l,m^{-1})$, we have that $|\D^{k-n}u|\leq L$, so $Z_{l,m,L}$ is closed. 

Next, we consider cut--off functions $\rho_{l,m}\in\hold^\infty_c$ such that $\rho_{l,m}=1$ in $B(x_l,m^{-1})$. It follows that $\rho_{l,m}u_h(\cdot-x_l)\in \bv^\A(\R^n)$ by the regularity of $u_h$ away from zero and the Leibniz rule (it is still the case that $u_h=K^\A w$, as in the proof of necessity for Theorem~\ref{thm:main_infty} from Section~\ref{sec:proof_infty}). Moreover, since $m\geq2$ and $x_l\in B$, one can easily arrange that the support of $\rho_{m,l}$ is contained in $2\bar B$, so $\rho_{l,m}u_h(\cdot-x_l)\in \mathfrak{X}$. However, $\D^{k-n}u_h(\cdot-x_l)\notin\lebe^{\infty}(\ball(x_l,m^{-1}))$, so that $Z_{l,m,\infty}$ is a proper subspace of $\mathfrak{X}$. Therefore, the sets $Z_{l,m,L}\subset Z_{l,m,\infty}$ have empty interior, so we can conclude by Baire's Category Theorem.
\end{proof}

\section{Appendix}

\subsection{Lorentz--differentiability of $\bv^\A$--maps}\label{sec:weak_est}
When presenting the main results on $\lebe^p$--Taylor expansions in the Introduction we refrained from presenting them on the full Lorentz $\lebe^{p,q}$--scale in order to not overcomplicate the statements. However, presenting the results in this manner reveals, yet again, a very clear connection between $\lebe^p_{\mathrm{weak}}$--differentiability and weak--type estimates as well as $\lebe^{p,q}$--differentiability, $q<\infty$, and strong--type estimates. 

We begin by extending the definition of the $\taylor^{k,p}$--spaces to fit the Lorentz scale:
\begin{definition}
Let $\Omega\subset\R^n$ be an open set, $x\in\Omega$, $1\leq p\leq\infty$, $1\leq q\leq \infty$. A measurable map $u\colon\Omega\rightarrow V$ lies in $\taylor^{k,(p,q)}(x)$ if there exists a polynomial $P_x^ku$ of degree at most $k$ such that
\begin{align*}
\|R^k_x u\|_{\lebe^{p,q}(\ball_r(x))}=o\left(r^{k+n/p}\right)\quad\text{ as }r\downarrow0,
\end{align*}
where $R^k_xu\coloneqq u-P^k_xu$. If this is the case, we say that $u$ has a \emph{$k$--th order $\lebe^{p,q}$--Taylor expansion} $P^k_x u$ at $x$.
\end{definition}
Here for $p=\infty$ we use the extension of the Lorentz scale, as introduced in \cite{BMB}, which is such that $\lebe^{\infty,1}_{\locc}\hookrightarrow\lebe^{\infty,q}_{\locc}\hookrightarrow\lebe^{\infty,\infty}_{\locc}$. Moreover,
\begin{align*}
\lebe^{\infty,1}=\lebe^\infty\qquad\text{ and }\qquad\lebe^{\infty,\infty}=\lebe^\infty_{\mathrm{weak}},
\end{align*}
where the weak--$\lebe^\infty$ space was introduced in \cite{BDVS}. We also recall the formal definition of the $\lebe^{\infty,q}$--spaces, which are not linear spaces in general (see also \cite{MP}).
\begin{definition}[{\cite[Def.~2.1]{BMB}}]
Let $\Omega\subset\R^n$, $|\Omega|\coloneqq\mathscr{L}^n(\Omega)$, $0<q\leq \infty$. We say that $f\in\lebe^{\infty,q}(\Omega)$ if $f\colon\Omega\rightarrow\R$ is measurable and the quantity
\begin{align*}
\|f\|_{\lebe^{\infty,q}(\Omega)}&\coloneqq\left(\int_0^{|\Omega|}\left[f^{\ast\ast}(t)-f^\ast(t)\right]^q\frac{\dif t}{t}\right)^{1/q}\qquad\text{for }q<\infty\\
\|f\|_{\lebe^{\infty,\infty}(\Omega)}&\coloneqq\sup_{0<t<|\Omega|}\left[f^{\ast\ast}(t)-f^\ast(t)\right]
\end{align*}
is finite. Here $f^{\ast}
$ denotes the \emph{decreasing rearrangement} of $f$,
\begin{align*}
f^\ast(t)\coloneqq\inf\{\lambda>0\colon\mathscr{L}^n\{x\in\Omega\colon|f(x)|>\lambda\}\leq t\}
\end{align*}
and $f^{**}$ denotes the \emph{maximal function}
\begin{align*}
f^{**}(t)\coloneqq\dfrac{1}{t}\int_0^tf^*(s)\dif s
\end{align*}
for $t\in (0,|\Omega|)$.
\end{definition}


We make the convention $n/0=\infty$. The methods of the present paper are readily adapted to give the following unified extensions of Theorems~\ref{thm:main_k=1},~\ref{thm:main_k_diff},~\ref{thm:sub_crit}, which, despite their rather intricate statements, tell the story of this paper best:
\begin{theorem}[Weak--type]
Let $\A$ be as in \eqref{eq:Ak}, $1\leq j\leq \min\{k,n\}$. The following are equivalent:
\begin{enumerate}
\item For all $u\in\bv^\A_{\locc}(\R^n)$, we have that 
\begin{align*}
\D^{k-j}u\in \taylor^{j,(n/(n-j),\infty)}(x)\quad\text{ for }\mathscr{L}^n\text{--a.e. }x\in\R^n.
\end{align*}
\item $\A$ is elliptic.
\end{enumerate}
\end{theorem}
\begin{theorem}[Strong--type]
Let $\A$ be as in \eqref{eq:Ak}, $1\leq q<\infty$. Then:
\begin{enumerate}
\item If $1\leq j\leq \min\{k,n-1\}$ and $1<q<\infty$, the following are equivalent:
\begin{enumerate}
\item[$(\mathrm{i})$] For all $u\in\bv^\A_{\locc}(\R^n)$, we have that
\begin{align*}
\D^{k-j}u\in \taylor^{j,(n/(n-j),q)}(x)\quad\text{ for }\mathscr{L}^n\text{--a.e. }x\in\R^n.
\end{align*}
\item[$(\mathrm{ii})$] $\A$ is elliptic and canceling.
\end{enumerate}
\vspace{3pt}
\item If $k\geq n$ and $1\leq q<\infty$, the following are equivalent:
\begin{enumerate}
\item[$(\mathrm{i})$] For all $u\in\bv^\A_{\locc}(\R^n)$, we have that
\begin{align*}
D^{k-n}u\in\taylor^{n,(\infty,q)}(x)\quad\text{ for }\mathscr{L}^n\text{--a.e. }x\in\R^n.
\end{align*}
\item[$(\mathrm{ii})$] $\A$ is elliptic and satisfies condition~\eqref{eq:mistery_cond2}.
\end{enumerate}
\end{enumerate}
\end{theorem}
The methods of the paper can indeed be adapted in a straightforward manner, except for two aspects: necessity of condition~\eqref{eq:mistery_cond2} for the strong--differentiability if $1<q<\infty$ (which follows from the computable fact that $\log|\cdot|\notin\lebe^{\infty,q}_{\locc}$ for $1\leq q<\infty$) and sufficiency of ellipticity for weak--differentiability if $j=n\leq k$, which follows easily from the following result, which is probably obvious for experts:
\begin{proposition}\label{prop:weak_est_n}
Let $\A$ as in \eqref{eq:Ak} be elliptic, of order $k\geq n>1$. Then
\begin{align*}
\|D^{k-n}u\|_{\lebe^{\infty,\infty}(\R^n,V\odot^{k-n}\R^n)}\lesssim\|\A u\|_{\lebe^1(\R^n,W)}
\end{align*}
for all $u\in\hold_c^\infty(\R^n,V)$.
\end{proposition}
The inequality follows from Formula~\eqref{eq:conv_ker} and the following bound:
\begin{align}\label{eq:riesz_n}
\|I_nf\|_{\lebe^{\infty,\infty}}\lesssim\|f\|_{\lebe^1}
\end{align}
for $f\in\hold^\infty_c(\R^n)$. Here $I_nf\coloneqq f\star\log|\cdot|$ is the analytic extension of the Riesz potential $I_\alpha$, $0<\alpha<n$, see e.g. \cite[Ch.~2]{Samko}.
\begin{proof}[Proof of \eqref{eq:riesz_n}]
By \cite[Eq.~(1.1)]{BMB},
\begin{align*}
\|I_nf\|_{\lebe^{\infty,\infty}}\lesssim\|DI_nf\|_{\lebe^{n,\infty}}=\|I_{n-1}f\|_{\lebe^{n,\infty}}\lesssim\|f\|_{\lebe^1},
\end{align*}
where for last inequality we applied the standard weak--type estimate for $I_{n-1}$.
\end{proof}
The case $n=1$ in Proposition~\ref{prop:weak_est_n} has to be ruled out if we want to use boundedness of $I_{n-1}$ between $\lebe^1$ and $\lebe^n_{\mathrm{weak}}$, but this is no restriction, since in that case Theorem~\ref{thm:main_infty} applies and the strong estimate is available (every elliptic operator satisfies condition~\eqref{eq:mistery_cond2} in odd dimensions).
\begin{remark}
We conclude by recalling \cite[Open~Prob.~8.3]{VS}, which can be reformulated as, whether for elliptic operators $\A$, is it the case that cancellation implies
\begin{align*}
\|D^{k-j}u\|_{\lebe^{n/(n-j),1}}\lesssim\|\A u\|_{\lebe^1}
\end{align*}
for $u\in\hold^\infty_c(\R^n,V)$. Here $1\leq j\leq \min\{n-1,k\}$. What the viewpoint we took in section indicates is that the result of Theorem~\ref{thm:main_infty} applies to the endpoint $j=n\leq k$, though we do not see a connection with the aforementioned problem.
\end{remark}
\subsection{Added--in remarks}
The following two facts became known to the author after the submission of this paper. However, their connection to the results presented here seems very strong and definitely rounds off the exposition.
\begin{remark}[On a question of H. Brezis]
	Just days after the submission of this paper, the author became aware of the recent preprint \cite{APR}, 
	where it is proved, among other results, that for a locally integrable scalar field $u$ on $\R^n$ such that its Laplacian $\Delta u$ is a Radon measure, we have that 
	\begin{align}\label{eq:APR}
		\Delta^{ac} u=\mathrm{tr}(\nabla^2u)\quad\mathscr{L}^n\text{--a.e.};
	\end{align}
	here we used the notation in Section~\ref{sec:prel_appdiff}. This fact and a standard property of approximate derivatives on level sets are used in \cite[Sec.~9]{APR} to answer a question of H. Brezis. Since the Laplacian operator is elliptic, our proof of Lemma~\ref{lem:algebra_Au} seems to give an alternative approach to prove \eqref{eq:APR}. Similar and related questions were also considered in \cite{MM,CV,V}.
\end{remark}
\begin{remark}[Continuity and $n$--th order operators]
	Very recently, the question concerning the continuity of $\bv^\A$--maps raised at the end of Section~\ref{sec:WA1} was solved by the author and A. Skorobogatova in the preprint \cite{RaSk}. 
	It was shown in \cite[Thm.~1.1]{RaSk} that an elliptic operator $\A$ of order $n$ is canceling if and only if the inclusion $\bv^\A(\R^n)\subset\hold_0(\R^n,V)$ holds. This is in sharp contrast with Theorem~\ref{thm:sobo}, where it is shown that $\A$ is weakly canceling if and only if $\bv^\A(\R^n)\subset\lebe^\infty(\R^n,V)$. In particular, for elliptic operators of oder $k\geq n$, we see that J. Van Schaftingen's canceling condition, which is equivalent with $\bv^\A(\R^n)\subset\dot{\sobo}{^{k-j,n/(n-j)}}(\R^n,V)$ for $0<j<n$, does \emph{not} appear on the Sobolev/Lipschitz scale at the endpoint $j=n$, but on the $\hold^m$--scale.
\end{remark}
\small{\subsection*{Acknowledgement} The author warmly thanks Jan Kristensen and Jean Van Schaftingen for helpful comments, suggestions, and stimulating discussions. The author also thanks the anonymous referee, whose comments significantly improved the quality of the manuscript.  }

\end{document}